\documentclass[12pt,a4paper]{amsart}
\usepackage[top=35mm, bottom=30mm, left=30mm, right=30mm]{geometry}
\usepackage{mathptmx}
\usepackage{mathrsfs}
\usepackage{verbatim}
\usepackage{url}
\usepackage[all]{xy}
\usepackage{xcolor}
\usepackage[colorlinks=true,citecolor=blue]{hyperref}
\usepackage{amsmath,amsthm}
\usepackage{amssymb}
\usepackage{enumerate}
\usepackage{graphicx}

\newtheorem{thm}{Theorem}[section]

\newtheorem{lem}[thm]{Lemma}
\newtheorem{prop}[thm]{Proposition}

\theoremstyle{definition}
\newtheorem{defin}[thm]{Definition}
\newtheorem{rem}[thm]{Remark}
\newtheorem{exa}[thm]{Example}

\numberwithin{equation}{section}
\begin{document}


\baselineskip=17pt


\title{Quasi-intermediate value theorem and outflanking arc theorem for plane maps}

\author[J.~Mai]{Jiehua Mai}
\address{School of Mathematics and Quantitative
Economics, Guangxi University of Finance and Economics, Nanning, Guangxi, 530003, P. R. China \&
 Institute of Mathematics, Shantou University, Shantou, Guangdong, 515063, P. R. China}
\email{jiehuamai@163.com; jhmai@stu.edu.cn}

\author[E.~Shi]{Enhui Shi*}
\thanks{*Corresponding author}
\address{School of Mathematics and Sciences, Soochow University, Suzhou, Jiangsu 215006, China}
\email{ehshi@suda.edu.cn}

\author[K.~Yan]{Kesong Yan}
\address{School of Mathematics and Statistics, Hainan Normal
University,  Haikou, Hainan, 571158, P. R. China}
\email{ksyan@mail.ustc.edu.cn}

\author[F.~Zeng]{Fanping Zeng}
\address{School of Mathematics and Quantitative
Economics, Guangxi University of Finance and Economics, Nanning, Guangxi, 530003, P. R. China}
\email{fpzeng@gxu.edu.cn}

\begin{abstract}
For a disk $D$ in the plane $\mathbb R^2$ and a plane map $f$,
we give several conditions on the restriction of $f$ to the boundary $\partial D$
of $D$ which imply the existence of a fixed point of $f$ in some specified domain in $D$. These conditions are similar to those appeared in the
intermediate value theorem for maps on the real line. As an application of the main results, we establish a fixed point theorem
for plane maps having an outflanking arc, which extends the famous theorem due to Brouwer: if $f$ is an orientation-preserving homeomorphism on the plane and has a periodic point, then it has a fixed point.
\end{abstract}


\maketitle

\pagestyle{myheadings} \markboth{J. H. Mai, E. H. Shi, K. S. Yan, F. P. Zeng }{Quasi-intermediate value theorem and outflanking arc theorem}

\section{Introduction}

The following result is due to Brouwer \cite{Bro12b}, which is usually called Brouwer's Lemma.
Some of its simple modern proofs  can be found in \cite{BF93, Bro84, Fa87}.

\begin{thm}[Brouwer's Lemma]\label{Brouwer lemma} Let $f: \mathbb{R}^2
\rightarrow \mathbb{R}^2$ be an orientation preserving
homeomorphism. If $f$ has a periodic point, then $f$ has a fixed
point.
\end{thm}

In fact, the conclusion of the theorem also holds if $f$  is orientation-preserving and has a nonwandering point (see e.g. \cite{Fa87, Ca06}).
This is a cornerstone of Brouwer Theory, which is fundamental in the study of surface homeomorphisms. Franks used Brouwer's Lemma nicely  to generalize Poincar\'e-Birkhoff theorem \cite{Fr88, Fra88}. Some variations of the theorem are key in several proofs of Brouwer's Plane Translation Theorem \cite{Bro12b, CS96,  Fr92, Gu94} and in the study of translation vectors for torus homeomorphisms \cite{Fra89, KK08}.
 An equivariant foliated version of Brouwer's Plane Translation Theorem was established by Le Calvez and  used to give a positive answer to Conley's conjecture in the case of surfaces \cite{Cal06}.
Le Calvez and Tal used Brouwer Theory to study the forcing theory of surface homeomorphisms \cite{CT}.
One may consult \cite{Gu08, Sa01, Sl88} for extensions of Brouwer's Plane Translation Theorem and consult \cite{Ca06} for an interesting
survey on Brouwer Theory and its applications. We should note that if the orientation preserving
condition is removed, then  the conclusion of Theorem \ref{Brouwer lemma} does not hold any more \cite{Bo81}.
\medskip

A natural question is to what extent we can generalize Theorem \ref{Brouwer lemma} to the case of plane maps.
The following example indicates that the answer to this question is subtle.

\begin{exa} Let $n \in \mathbb{N}-\{1\}$,
and let $g: \mathbb{R} \rightarrow \mathbb{R}$ be a continuous map
defined by the following two conditions:

\medskip (1) $g(r)=r+1$ for any $r \in (-\infty, n-1]$;

\medskip (2) $g(n)=1$, and $g|_{[n-1, \infty]}$ is linear.


\medskip
\noindent Define $F: \mathbb{R}^2 \rightarrow \mathbb{R}^2$ by
$$F(r,s)=(g(r), s+d(r,\mathbb{Z}))$$
for any $(r,s) \in \mathbb{R}^2$.  Then $(1, 0)$ is an $n$-periodic point of $F$.
It is easy to see that $F$ has no $m$-periodic point for any $m
\in \mathbb{N}-\{n\}$. Specially, $F$ has no fixed point.

\end{exa}

The aim of the paper is to find some weaker conditions than that in Brouwer's Lemma, which ensure
the existence of fixed points for continuous maps on the plane. In Section 3,
 we establish a quasi-intermediate value theorem for plane maps (Theorem \ref{main}). Roughly speaking, for a disk $D$ in the plane $\mathbb R^2$ and a plane map $f$,
we give several conditions on the restriction of $f$ to the boundary $\partial D$
of $D$ which imply the existence of a fixed point of $f$ in some specified domain in $D$. These conditions are similar to those appeared in the
intermediate value theorem for maps on the real line; this is why we call the main result in this section quasi-intermediate
value theorem. Based on the results in Section 3, we introduce in Section 4 the notion of
outflanking arc (Definition \ref{step arc}) and show that if $f$ is a plane map having an outflanking arc with some technical
conditions (mainly a local injectivity and orientation preserving condition, and an exclusivity condition; see Theorem \ref{main-2} for details), then $f$ has a fixed point.
Then we show that an orientation preserving plane homeomorphism $f$ with a periodic point of period $\geq 2$
satisfies the conditions just mentioned (Theorem \ref{main3}). Thus Brouwer's Lemma is a direct corollary of these results.
To illustrate that our conditions are strictly weaker than that in Brouwer's Lemma, we construct
an orientation preserving homeomorphism of the plane with an outflanking arc but with no periodic points of period $\geq 2$.

\section{Preliminaries}

In this section, we will introduce some notions and notations, and facts mainly focusing on plane topology, which will be used in the
sequel of the paper.

\subsection{General notions and notations}
Throughout the paper, the real line $\mathbb R$ is regarded as a subspace of the plane $\mathbb R^2$,
that is, $\mathbb R=\mathbb R\times \{0\}\subset \mathbb R^2$. Let $0=(0, 0)$ be the origin of
$\mathbb R^2$. Let $d=d_E$ be the Euclidean metric on $\mathbb R^2$. For $x\in\mathbb R$, $X, Y\subset \mathbb R^2$, and
$r\geq 0$, write $B(x, r)=\{y\in\mathbb R^2: d(y, x)\leq r\}$, $B(X, r)=\bigcup\{B(x, r): x\in X\}$, and $d(X, Y)=\inf\{d(x, y): x\in X, y\in Y\}$.
Denote by ${\rm diam}(X)\equiv\sup\{d(y, z):y,\ z\in X\}$ the {\it\textbf{diameter}} of $X$.
A space $Y$ is called a {\it\textbf{disc}} (resp. a {\it \textbf{circle}}, resp. an {\it \textbf{arc}}) if
$Y$ is homeomorphic to the unit disc ${\mathbb D}\equiv B(0, 1)$ (resp. to the unit circle $\mathbb S^1\equiv \partial {\mathbb D}$,
resp. to the unit interval $[0, 1]$). If $J$ is an arc and $h:J\rightarrow [0, 1]$ is a homeomorphism,
then $\partial J\equiv\{h^{-1}(0), h^{-1}(1)\}$ is called the {\it \textbf{endpoint set}} of $J$; we use
$\stackrel{\circ} {J}$ to denote $J-\partial J.$ (Notice that we also use $\stackrel{\circ} {Y}$ to denote the
interior ${\rm Int}\ Y$ of $Y$ for any $Y\subset\mathbb R^2$.)
For any circle $C$ in
$\mathbb{R}^2$, denote by $D(C)$ the disc in $\mathbb{R}^2$ such
that $\partial D(C)=C$. For any open set $U$ in $\mathbb{R}^2$,
write ${\rm Bd}(U)=\partial U=\overline{U}-U$. For any $x, y \in \mathbb{R}^2$ with
$x \neq y$, let $[x,y]$ denote the straight line segment with
endpoints $x$ and $y$, let $(x,y)=[x,y]-\{x,y\}$, and write
$|x-y|=d_E(x,y)$. For any arc $A$ and any $x,y \in A$, let $[x,y]_A$
denote the smallest connected subset of $A$ containing $\{x,y\}$ and
write $(x,y]_A=[y,x)_A=[x,y]_A-\{x\}$ and $(x,y)_A=[x,y]_A-\{x,y\}$.
Note that if $x=y$ then $(x,y]_A=(x,y)_A=\emptyset$.
\medskip

Let $X$ and $Y$ be two subspaces of some metric space, $V \subset E
\subset X$, $W \subset X \cup Y$, and $f: X \rightarrow Y$ be a
continuous map. Let $F|V: X \rightarrow Y$ be the restriction of $f$
to $V$. Define $f||V: V \rightarrow f(V)$ by $(f||V)(x)=f(x)$ for
any $x \in V$, called the {\it \textbf{bi-restriction}} of $f$ to
$V$. By the definition, $f||V$ must be a surjection, and if $f|V$ is
an injection then $f||V$ is a bijection. If $V$ is compact and $f|V$
is an injection then $f||V$ is a homeomorphism. If
$f(E-V) \cap f(V)=\emptyset$, then we say that $f|V$ is {\it
\textbf{exclusive}} in $E$. If $f(V) \cap W=\emptyset$, then we say
that $f|V$ {\it \textbf{dodges}} $W$. Specially, if $V \subset X
\cap Y$ and $f(V) \cap V=\emptyset$ then we say that $f|V$ is {\it
\textbf{moving}}. As examples, we have: (1) If $f|E$ is injective,
then $f|V$ is exclusive in $E$; (2) If $f(v) \neq v$ for some $v \in
V$, then there is a neighborhood $U$ of $v$ in $V$ such that $f|U$
is moving.

\medskip Let $X$ be a  metric space and $f:X\rightarrow X$ be a
continuous map.  Let\;\! $\mathbb N$ \;\!be the set
of all positive integers \,and let\;\! $\mathbb Z_+=\mathbb
N\cup\{0\}$. For any \;\!$n\in\mathbb N$\;\!, \,write \;\!$\mathbb
N_n=\{1,\cdots , n\}$. Let
$f^{\;\!0}$ be the identity map of $X$, and let $f^{\;\!n}=f\circ
f^{\;\!n-1}$ be the composition map of $f$ and $f^{\;\!n-1}$.
For $x\in X$, the set \;\!$O(x,f)\equiv\{f^{\,n}(x): n\in \Bbb Z_+\}$ \,is called the
{\it\textbf{orbit}}\, of \;\!$x$\;\! under $f$. A
point $x\in X$  is called a {\it\textbf{periodic
point}}\, of $f$ if $f^n(x)=x$ \, for some $n\in\mathbb N$\;\!; the minimal such $n$ is called
the {\it\textbf{order}} of $x$. If $x$ is a periodic point of order $1$, that is $f(x)=x$, then
$x$ is called a {\it\textbf{fixed point}} of $f$.  Denote by ${\rm Fix}(f)$\;\! and
${\rm Per}(f)$\;\! the sets of fixed points and periodic points, respectively.

\subsection{Complementary domains of a Peano continuum in the plane}

Recall that a {\it\textbf{continuum}} is a connected compact metric space and a {\it \textbf{Peano continuum}} is a locally connected continuum.
It is well known that if $X$ is a Peano continuum, then it is arcwise connected; that is,  for any $x\not=y\in X$, there exits an arc
$J$ in $X$ with $\partial J=\{x, y\}$ (see \cite[Theorem 8.26]{Na92}).
\medskip

The conclusion (C.1.1) of the following lemma can be seen in \cite[Chapter 10-Theorem 11]{Ku68} and (C.1.2) is
clearly implied by (C.1.1).

\begin{lem}\label{arc}
Let $P$ be a Peano continuum in $\mathbb R^2$ and $V$ be a connected
component of $\mathbb{R}^2-P$. Suppose $w \in V$ and  $z, z' \in
Bd(V)$ with $z \neq z'$. Then

\medskip {\rm (C.1.1)}\ There is an arc $L$ such that $\partial
L=\{w,z\}$ and $L-\{z\} \subset V$.

\medskip {\rm (C.1.2)}\ There is an arc $\Gamma$ such that
$\partial \Gamma=\{z, z'\}$ and $\stackrel{\circ}\Gamma \subset V$.
\end{lem}

\subsection{ Extensibility of homeomorphisms between $\sigma$-graphs}

In the study of arcs and circles in the plane, the following two
well-known theorems are very useful (see \cite[pp. 20-32]{Bing83} or
\cite[pp. 65-70]{Moise77}).

\begin{thm}\label{arc nonsep}  No arc in the plane $\mathbb{R}^2$
separates $\mathbb{R}^2$.
\end{thm}

\begin{thm} [The Schoenfiles Theorem] \label{circle extension}Any
homeomorphism between two circles in the plane $\mathbb{R}^2$ can be
extended to a homeomorphism of $\mathbb{R}^2$.
\end{thm}

 A graph homeomorphic to the subspace $\mathbb S^1 \cup [1, 2]$
of $\mathbb{R}^2$ is called a {\it\textbf{ $\sigma$-graph}}. By Theorems \ref{arc nonsep} and
\ref{circle extension}, it is easy to show the following theorems.

\begin{thm} Any homeomorphism between two arcs in the
plane $\mathbb{R}^2$ can be extended to a homeomorphism of
$\mathbb{R}^2$.
\end{thm}

\begin{thm}\label{sigma} Any homeomorphism between two
$\sigma$-graphs in the sphere $\mathbb{R}^2 \cup \{\infty\}$ can be
extended to a homeomorphism of $\mathbb{R}^2 \cup \{\infty\}$.
\end{thm}

\subsection{Topological conjugacy}

Let $X_1, X_2$ and $Y_1, Y_2$ be subsets of $\mathbb{R}^2$. Suppose
that $f: X_1 \rightarrow X_2$ and $g: Y_1 \rightarrow Y_2$ are
continuous maps. If there exists a homeomorphism $h: \mathbb{R}^2
\rightarrow \mathbb{R}^2$ such that $h(X_1)=Y_1$, $h(X_2)=Y_2$ and
$g \circ h|X_1=h \circ f$ (equivalently, $g=h \circ f \circ
h^{-1}|Y_1$, or $f=h^{-1} \circ g \circ h|X_1$), then $f$ and $g$
are said to be {\it \textbf{topologically conjugate}}, and $h$ is
called a {\it \textbf{topological conjugacy}} from $f$ to $g$. A
property of a continuous map $f: X_1 \rightarrow X_2$ is said to be
{\it \textbf{invariant under topological conjugacy}} if each
continuous map $g: Y_1 \rightarrow Y_2$ topologically conjugate to
$f$ also has this property. Clearly, the property that a continuous
map $f: X_1 \rightarrow X_2$ has a fixed point or has an
$n$-periodic point for some $n \in \mathbb{N}$ is invariant under
topological conjugacy.

\begin{rem} \label{top conj} Topological conjugacy can be used in the
proofs of some theorems. For example, suppose that there is a
theorem which mentions that a continuous map $f: X_1 \rightarrow
X_2$ has a property (P). Suppose that the property (P) is invariant
under topological conjugacy, and there is a circle $C$ in $X_1$,
which is only required to be homeomorphic to the unit circle $\mathbb S^1$.
However, in proving this theorem, by means of topological conjugacy,
we can assume that $C$ itself is the unit circle $\mathbb S^1$, or is the
boundary of a square. Similarly, there is a $\sigma$-graph $G$ in
$X_1$, then, by Theorem \ref{sigma}, in proving this theorem we can assume
that the $\sigma$-graph $G$ has a specific shape.
\end{rem}

\subsection{Local injectivity and injectivity}

\begin{lem} \label{injectivity} Let $X, Y$ be two metric spaces, $V$ be a
compact subset of $X$, and $f: X \rightarrow Y$ be a continuous map.
Then the following two properties are equivalent:

\medskip

(a)\, $f|V$ is injective, and for each $x \in V$ there exists a
neighborhood $U_x$ of $x$ in $X$ such that $f|U_x$ is injective;

\medskip

(b)\,  There exists a neighborhood $U$ of $V$ in $X$ such that
$f|U$ is injective.

\end{lem}

\begin{proof} Clearly, property (b) implies (a). Thus it
suffices to show (a) implies (b). For each $x \in V$, there exists
$\varepsilon(x)>0$ such that $f|B(x, 3 \varepsilon(x))$ is
injective. Since $V$ is compact, there exist $x_1, x_2, \ldots, x_n$
in $V$ such that the family $\{B(x_k, \varepsilon(x_k): k \in
\mathbb{N}_n\}$ is a cover of $V$. Let
$\varepsilon=\min\{\varepsilon(x_k): k \in \mathbb{N}_n\}$. Then for
any $x \in V$, there exists a $k=k(x) \in \mathbb{N}_n$ such that
$$x \in B(x_k, \varepsilon(x_k)), \ \ B(x, \varepsilon) \subset B(x_k, 2 \varepsilon(x_k)),$$
and hence $f|B(x, \varepsilon)$ is injective. Let
$$\delta'=\min\left\{d(f(x), f(y)): x, y \in V \mbox{\ with\ } d(x,y) \geqslant \varepsilon\right\}.$$
Then $\delta'>0$ since $f|V$ is injective. Take $\delta \in (0,
\varepsilon]$ such that
$$f(B(x, \delta)) \subset B\left(f(x), \frac{\delta'}{3}\right) \mbox{\ \ for any\ } x \in V.$$
Then for any $x, y \in V$, if $d(x,y) \geqslant \varepsilon$ then
$$d(f(B(x,\delta)), f(B(y, \delta)))\geq d(f(x), f(y))-\frac{2\delta'}{3}>0,$$
which implies that $f|(B(x,\delta) \cup B(y, \delta))$ is injective.
If $d(x,y)<\varepsilon$, then there exists $k \in \mathbb{N}_n$ such
that $x \in B(x_k, \varepsilon(x_k))$, and hence
$$B(x,\delta) \cup B(y, \delta) \subset B(x_k,
3\varepsilon(x_k)),$$which also implies that $f|(B(x,\delta) \cup
B(y,\delta))$ is injective. Therefore, take
$$U=\bigcup\{B(x, \delta): x \in V\},$$
then $U$ is a neighborhood of $V$ in $X$, and $f|U$ is injective.
Lemma \ref{injectivity} is proved.
\end{proof}

\subsection{Orientations of manifolds}

 Let $M$ be an orientable
closed $n$-manifold with $n\geq 1$. Then the $n$-homology group $H_n(M, \mathbb{Z})$ is
isomorphic to $\mathbb{Z}$. Either generator of $H_n(M, \mathbb{Z})$
is called an {\it \textbf{orientation}} of $M$. Let $h: M
\rightarrow M$ be a homeomorphism. Then $h$ induces an isomorphism
$h_{*}: H_n(M, \mathbb{Z}) \rightarrow H_n(M, \mathbb{Z})$. If
$h_{*}$ is the identity map, then $h$ is said to be {\it
\textbf{orientation preserving}}; if $h_{*}(s)=-s$ for any $s
\in H_n(M, \mathbb{Z})$, then $h$ is said to be {\it
\textbf{orientation reversing}}. Clearly, if
$f, g: M \rightarrow M$ are topologically conjugate homeomorphisms,  then both $f$ and $g$ are
either orientation preserving, or  orientation
reversing. One may consult \cite{We14} for the details about orientations of manifolds.
\medskip

Let $h: \mathbb{R}^n \rightarrow \mathbb{R}^n$ be a homeomorphism.
Then $h$ can be uniquely extended to a homeomorphism $h':
\mathbb{R}^n \cup \{\infty\} \rightarrow \mathbb{R}^n \cup
\{\infty\}$ of the $n$-sphere $\mathbb{R}^n \cup \{\infty\}$ by
putting $h'(\infty)=\infty$. $h$ is said to be {\it
\textbf{orientation preserving}} (resp., {\it \textbf{orientation
reversing}}) if $h'$ is orientation preserving (resp., orientation
reversing).

\begin{defin}\, Let $V$ be a circle or a disc in
$\mathbb{R}^2$, and $\xi: V \rightarrow \mathbb{R}^2$ be a
continuous injection. Then $\xi(V)$ is also a circle or a disc, and
$\xi||V: V \rightarrow \xi(V)$ is a homeomorphism. By Theorem \ref{circle extension},
$\xi$ can be extended to a homeomorphism $\xi': \mathbb{R}^2
\rightarrow \mathbb{R}^2$. If $\xi'$ is orientation preserving
(resp., orientation reversing), then we say that the homeomorphism
$\xi||V: V \rightarrow \xi(V)$ or the continuous injection $\xi: V
\rightarrow \mathbb{R}^2$ is {\it \textbf{orientation preserving}}
(resp., {\it \textbf{orientation reversing}}).
\end{defin}

By Theorem \ref{circle extension} with some techniques we can prove the following
properties:

\medskip

(P.1)\ It is independent of the choice of the extension $\xi':
\mathbb{R}^2 \rightarrow \mathbb{R}^2$ whether $\xi$ is orientation
preserving or not.

\medskip

(P.2)\ If the continuous injection $\xi: V \rightarrow \mathbb{R}^2$
is orientation preserving, then $\xi$ and the identity imbedding
$\mathrm{id}_V: V \rightarrow \mathbb{R}^2$ are isotopic, that is,
there exists a family $\{\xi_t: V \rightarrow \mathbb{R}^2~|~t \in
[0,1]\}$ of continuous injections such that $\xi_0=\mathrm{id}_V$,
$\xi_1=\xi$, and $\xi_t$ is continuously dependent on $t$.

\medskip

(P.3)\ Let $W=\mathbb D$\ (if $V$ is a disc) or $W=\mathbb S^1$ (if $V$ is a
circle), and $\eta: W \rightarrow V$ be a homeomorphism. Define the
reflection $\varphi: W \rightarrow W$ by $\varphi(r,s)=(-r,s)$ for
any $(r,s) \in W$. Then the continuous injection $\xi: V \rightarrow
\mathbb{R}^2$ is orientation preserving (resp., orientation
reversing) if and only if $\xi \eta \varphi \eta^{-1}: V \rightarrow
\mathbb{R}^2$ is orientation reversing (resp., orientation
preserving).

\medskip

In addition, by the above definition and property (P.1) we have

\medskip

(P.4)\ Let $V$ be a disc in $\mathbb{R}^2$, and $C=\partial V$. Then
the continuous injection $\xi: V \rightarrow \mathbb{R}^2$ is
orientation preserving (resp., orientation reversing) if and only if
$\xi|C: C \rightarrow \mathbb{R}^2$ is orientation preserving
(resp., orientation reversing).

\medskip

(P.5)\ Let $Y \subset \mathbb{R}^2$, $U$ be a connected subset of
$Y$ which is open in $\mathbb{R}^2$, and $f: Y \rightarrow
\mathbb{R}^2$ be a continuous map. If $f|U$ is injective, then for
any two discs $D$ and $E$ in $U$, $f|D$ is orientation preserving
(resp., orientation reversing) if and only if $f|E$ is orientation
preserving (resp., orientation reversing).

\medskip

Let $Y, U$ and $f$ be as in (P.5). If $f|U$ is injective, and for a disc $D \subset
U$, $f|D$ is orientation preserving (resp., orientation reversing),
then we say that $f|U$ is {\it \textbf{orientation preserving}}
(resp., {\it \textbf{orientation reversing}}).

\subsection{Directed arcs, directed angles, and directed circles}
An arc $A$ with a bijection $\tau: \{0, 1\} \rightarrow \partial A$
is called a {\it \textbf{directed arc}}, written $(A, \tau)$. The
points $\tau(0)$ and $\tau(1)$ are called the {\it \textbf{starting
point}} and the {\it \textbf{ending point}} of $(A, \tau)$,
respectively.

\medskip

Let $D$ be a disc, and $C=\partial D$. A continuous map $\rho: [0,1]
\rightarrow C$ is called a {\it \textbf{quasi-covering map}} if the
following condition holds:

\medskip

(C.1)\ $\rho|[0,1)$ is injective, and $\rho(0)=\rho(1)$.
\medskip
\\
Note that if $\rho: [0,1] \rightarrow C$ is a quasi-covering map
then the map $\rho_{\mathbb{R}}: \mathbb{R} \rightarrow C$ defined
by $\rho_{\mathbb{R}}(n+r)=\rho(r)$ for any $n \in \mathbb{Z}$ and
any $r \in [0,1)$ is a covering map. The circle $C$ with a
quasi-covering map $\rho: [0,1] \rightarrow C$ is called a {\it
\textbf{directed circle}} and is written $(C, \rho)$. The disc $D$
with the quasi-covering map $\rho: [0,1] \rightarrow C$ is called a
{\it \textbf{directed disc}} and is written $(D, \rho)$. Let $\rho':
[0,1] \rightarrow C$ be also a quasi-covering map. If there exists
an isotopy $\{\rho_t: [0,1] \rightarrow C~|~t \in [0,1]\}$ consists
of quasi-covering maps such that $\rho_0=\rho$ and $\rho_1=\rho'$,
and then we regard directed circles $(C, \rho)$ and $(C, \rho')$ as
the same.

\medskip

Define $\overleftarrow{\rho}: [0,1] \rightarrow C$ by
$\overleftarrow{\rho}(r)=\rho(1-r)$ for any $r \in [0,1]$, called
the {\it \textbf{inverse path}} of $\rho$. It is easy to show that
$\rho$ and $\overleftarrow{\rho}$ are not isotopic (in the family of
all quasi-covering maps from $[0,1]$ to $C$), and if $\rho'$ and
$\rho$ are not isotopic, then $\rho'$ and $\overleftarrow{\rho}$ are
isotopic. Thus there are only two directed circles based on $C$,
which are $(C, \rho)$ and $(C, \overleftarrow{\rho})$.

\medskip

We regard the complex plane $\mathbb{C}$ and the real plane
$\mathbb{R}^2$ as the same. In detail, for any $r>0$ and $\theta \in
\mathbb{R}$, the point $r \mathrm{e}^{\mathrm{i}\theta}=r \cos
\theta + \mathrm{i} \cdot r \sin \theta \in \mathbb{C}$ and the
point $(r \cos \theta, r \sin \theta) \in \mathbb{R}^2$ will be
regarded as the same. For any three different points $v$, $x$, $y$
in $\mathbb{C}$ with $|x-y|<|x-v| + |v-y|$, take real numbers
$\theta_{vx}$ and $\theta_{vy}$ such that

\medskip

(C.2)\ $(x-v)/|x-v|=\mathrm{e}^{\mathrm{i} \theta_{vx}}$,
$(y-v)/|y-v|=\mathrm{e}^{\mathrm{i}\theta_{vy}}$, and
$-\pi<\theta_{vy}-\theta_{vx}<\pi$,
\medskip
\\
and we define
$$\angle xvy=\theta_{vy}-\theta_{vx},$$
called the {\it \textbf{directed angle}} from the ray
$\overrightarrow{vx}$ to the ray $\overrightarrow{vy}$. Note that
$\angle xvy$ is independent of the choice of the pair of real
numbers $\theta_{vx}$ and $\theta_{vy}$ satisfying the condition
(C.2).

\medskip

Let $\rho: [0,1] \rightarrow \mathbb{C}=\mathbb{R}^2$ be a
continuous map, and $v \in \mathbb{R}^2-\rho([0,1])$. Take
$\delta>0$ such that, for any $r, s \in [0,1]$, if $|r-s|<\delta$
then
$$|\rho(r)-\rho(s)|<\frac{d_{E}(v, \rho([0,1]))}{3}.$$
For any given $a, b \in [0,1]$ with $0 \leqslant a<b \leqslant 1$,
take a sequence $r_0, r_1, \ldots, r_n$ such that \medskip

(C.3)\ $r_0=a$, $r_n=b$, and $0<r_k-r_{k-1}<\delta$ for each $k \in
\mathbb{N}_n$,
\medskip
\\
and we define
$$\angle (\rho(a)v \rho(b), \rho)=\sum_{k \in \mathbb{N}_n} \angle \rho(r_{k-1})v \rho(r_k),$$
called the {\it \textbf{rotational angle of ray
$\overrightarrow{v\rho(r)}$ along the path $\rho: [0,1] \rightarrow
\mathbb{C}$ from $\overrightarrow{v\rho(a)}$ to
$\overrightarrow{v\rho(b)}$}}. Note that

\medskip

(i)\ $\angle(\rho(a)v\rho(b), \rho)$ is independent of the choice of
sequence $r_0, r_1, \ldots, r_n$ satisfying the condition (C.3);

\medskip

(ii)\ $(\angle(\rho(a) v \rho(b), \rho)-\angle
\rho(a)v\rho(b))/(2\pi)$ is an integer. Specially, if there is a
straight line $L$ passing $v$ such that $L \cap
\rho([a,b])=\emptyset$, then $\angle(\rho(a)v\rho(b), \rho)=\angle
\rho(a)v\rho(b) \in (-\pi, \pi)$.

\medskip

Let $D$ be a disc in $\mathbb{R}^2$, $C=\partial D$, and let $\rho:
[0,1] \rightarrow C$ be a quasi-covering map. Then for any $v, v' \in
\stackrel{\circ}{D}$ , we have $\angle(\rho(0)v \rho(1), \rho)=
\angle(\rho(0)v'\rho(1), \rho) \in \{2\pi, -2\pi\}$.\ If
$\angle(\rho(0)v \rho(1), \rho)$ $=2\pi$ (resp., $-2\pi$), then the
directed circle $(C, \rho)$ and the directed disc $(D, \rho)$ are
said to be {\it \textbf{overall anticlockwise}} (resp., {\it
\textbf{overall clockwise}}). Note that, even if $\angle (\rho(0)v
\rho(1), \rho)=2\pi$ (resp., $-2\pi$), it is still possible that
there exist $0<a<b<1$ such that $\angle(\rho(a)v\rho(b), \rho)<0$
(resp., $>0$).
By the above properties (P.3) and (P.4) we can prove the following

\medskip

\begin{prop} \label{wind} Let $C$ be a circle in
$\mathbb{R}^2$, $\xi: C \rightarrow \mathbb{R}^2$ be a continuous
injection, and $\rho: [0,1] \rightarrow C$ be a quasi-covering map.
Then  $\xi: C \rightarrow \mathbb{R}^2$ is orientation
preserving {\rm(}resp., orientation reversing{\rm )} if and only if,
for any $v \in \mathrm{Int}\ D(C)$ and any $w \in \mathrm{Int}\
D(\xi(C))$,
$$\angle(\rho(0)v \rho(1), \rho) \cdot \angle(\xi \rho(0) w \xi\rho(1), \xi \rho)=4\pi^2 \mbox{\ (resp., $-4\pi^2$)}.$$
\end{prop}

\subsection{ Sides  of directed circles
and directed arcs}

\begin{defin}\label{side} (1) Let $D$ be a disc in $\mathbb{R}^2$,
$C=\partial D$, and $\rho: [0,1] \rightarrow C$ be a quasi-covering
map. If the directed circle $(C, \rho)$ is overall anticlockwise
(resp., overall clockwise), then any point in $\stackrel{\circ}D$ is
said to be {\it \textbf{on the left side}} (resp., {\it \textbf{on
the right side}}) of $(C, \rho)$, and any point in $\mathbb{R}^2-D$
is said to be {\it \textbf{on the right side}} (resp., {\it
\textbf{on the left side}}) of $(C, \rho)$.

\medskip

(2) Let $A$ be an arc in $\mathbb{R}^2$, $\tau: \{0, 1\} \rightarrow
\partial A$ be a bijection, $D$ be a disc in $\mathbb{R}^2$, and $C=\partial
D$. Suppose that $\partial A=\{x, y\}$, $\tau(0)=x$, $\tau(1)=y$,
and there exist $u \in [x,y)_A$ and $v \in (u, y]_A$ such that
\begin{equation} \label{eq:2-1}
D \cap A=C \cap A=[u, v]_A.
 \end{equation}
Take a quasi-covering map $\rho: [0,1] \rightarrow C$ such that
$$C \cap A=\rho([r,s]), \rho(r)=u, \rho(s)=v$$
for some $0<r<s<1$. If the directed circle $(C, \rho)$ is overall
anticlockwise (resp., overall clockwise), then we say that the disc
$D$ is {\it \textbf{on the left side}} (resp., {\it \textbf{on the
right side}}) of the directed arc $(A, \tau)$.

\medskip

In addition, for any path connected set $V$ in $\mathbb{R}^2$, if
the following condition holds:

\medskip

\noindent($\#$)\, $V \cap \stackrel{\circ}A \neq \emptyset$, $V-A \neq
\emptyset$, and there is a disc $D$ in $\mathbb{R}^2$ satisfying
\eqref{eq:2-1} such that $V \subset D$,
\medskip
\\
then we also say that $V$ is  {\it \textbf{on the left side}}
(resp., {\it \textbf{on the right side}}) of the directed arc $(A,
\tau)$ if $D$ is on the left side (resp., on the right side) of the
directed arc $(A, \tau)$.

\end{defin}

Note that, in (2) of Definition \ref{side}, only the discs $D$ satisfying
\eqref{eq:2-1} and the path connected sets $V$ satisfying the
condition ($\#$) are considered. For any compact set $W$ in
$\mathbb{R}^2-A$, in (2) of Definition \ref{side}, we have not give a
definition for $W$ to be on the left side or on the right side of
the directed arc $(A, \tau)$.

\medskip

From Proposition \ref{wind} we can obtain the following

\begin{prop} Let $E$ be a disc in
$\mathbb{R}^2$, and $f: E \rightarrow \mathbb{R}^2$ be a continuous
injection.

\medskip

(1)\, {\it Let $C$ be a circle in $\stackrel{\circ}E$, $w \in E-C$,
and $\rho: [0,1] \rightarrow C$ be a quasi-covering map. If $f$ is
orientation preserving, then $w$ is on the left side {\rm (}resp.,
on the right side {\rm )} of the directed circle $(C, \rho)$ if and
only if $f(w)$ is on the left side {\rm (}resp., on the right side
{\rm )} of the directed circle $(f(C), f \rho)$.}

\medskip

{\it If $f$ is orientation reversing, then $w$ is on the left side
{\rm (}resp., on the right side {\rm )} of the directed circle $(C,
\rho)$ if and only if $f(w)$ is on the right side {\rm (}resp., on
the left side {\rm )} of the directed circle $(f(C), f \rho)$.}

\medskip

(2)\, {\it Let $A$ be an arc in $\stackrel{\circ}E$, $D$ a disc in
$E$, and $\tau: \{0,1\} \rightarrow \partial A$ be a bijection. If
$f$ is orientation preserving, then $D$ is on the left side {\rm
(}resp., on the right side {\rm )} of the directed arc $(A, \tau)$
if and only if $f(D)$ is on the left side {\rm (}resp., on the right
side {\rm )} of the directed arc $(f(A), f \tau)$.}

\medskip

{\it If $f$ is orientation reversing, then $D$ is on the left side
{\rm (}resp., on the right side {\rm )} of the directed arc $(A,
\tau)$ if and only if $f(D)$ is on the right side {\rm (}resp., on
the left side {\rm )} of the directed arc $(f(A), f \tau)$.}

\end{prop}

\section{Quasi-Intermediate Value Theorem}

The following fixed point theorem is well known. Since it is equivalent to the Intermediate Value Theorem,
we also call it  Intermediate Value Theorem here.

\begin{thm}[Intermediate Value Theorem]\label{mean value} Let $I$ be a connected
subset of $\mathbb{R}$, and $f: I \rightarrow \mathbb{R}$ be a
continuous map. If there exist $x,y$ in $I$ such that
\begin{equation} \label{eq:3-1}
\big(f(x)-x\big) \cdot \big(f(y)-y\big)\leq 0,
 \end{equation}
then $f$ has a fixed point in $[x,y]$\;\!.
\end{thm}

The condition ($\ref{eq:3-1}$) is about the restriction of $f$ to the boundary of $[x, y]$, which  ensures the existence of fixed points of $f$ in $[x, y]$.  The aim of this section is to establish a fixed theorem for plane maps $f$,  in the same spirit as in Theorem \ref{mean value}; that is, for a disc $X$ in the plane, we  will give some conditions on the restriction of $f$ to the boundary $\partial X$ which ensure the existence of
 fixed points of $f$ in some specified domain in $X$.
\medskip

Let $\lambda: A \rightarrow [0,1]$ be a homeomorphism. For any $x, y \in A$, write
$x \prec y$ or $y \succ x$ if $\lambda(x)<\lambda(y)$. This relation
$\prec$ is called the {\it \textbf {order}} on $A$ induced by
$\lambda$. In Theorem \ref{mean value}, if we write $u=f(x)$ and $v=f(y)$, and assume that $x<y$, then ($\ref{eq:3-1}$)
holds  if and only if one of the following conditions holds:

(1) $\{u, v\} \subset [x,y]$;

(2) $x \leq v<y<u$;

(3) $v<x<u \leq y$;

(4) $v<x<y<u$;

(5) $u<x<y<v$.

\noindent These conditions inspire us to give the following fixed point theorem.
\medskip

\begin{thm}[Quasi-Intermediate Value Theorem]\label{main}  Let $X$ be a disc in
$\mathbb{R}^2$, $f:X\rightarrow \mathbb R^2$ be a continuous map,  and $A$ be an arc in $\partial X$.
Let $\lambda: A \rightarrow [0,1]$ be a homeomorphism and $\prec$
be the order on $A$ induced by $\lambda$. Suppose there exist points
$x, y, u, v$ in $A$ such that $x \prec y$, $f(x)=u$, $f(y)=v$,
\begin{equation}\label{eq:3-2}
\big(\lambda(u)-\lambda(x)\big) \cdot \big(\lambda(v)-\lambda(y)\big)<0,
\end{equation}
\begin{equation}\label{eq:3-2-1}
f\big([x,y]_A\big)\subset X,\ \ f\big([x,y]_A\big) \cap (u, v)_A=\emptyset, \ \  f^{-1}f\big((x,y)_A\big)=(x,y)_A,
\end{equation}
 and there exist a bounded connected
component $W$ of $\mathbb{R}^2-f\big([x,y]_A\big)-[u, v]_A$ and a neighborhood $U_q$
of some point $q \in (x,y)_A$ in $X$ such that $f(U_q-A) \subset W$.
Then $f$ has a fixed point in $W$ if one of the following conditions
holds:

\medskip$(1)$\ $\{u,v\} \subset [x,y]_A$;

\medskip$(2)$\ $x \preceq v \prec y \prec u$, and
$f\big([y,u]_A\big) \cap [y,u]_A=\emptyset$;

\medskip$(3)$\ $v \prec x \prec u \preceq y$, and
$f\big([v,x]_A\big) \cap [v,x]_A=\emptyset$;

\medskip$(4)$\ $v \prec x \prec y \prec u$, and
$f\big([v,x]_A\big) \cap [v,x]_A=f\big([y,u]_A\big) \cap
[y,u]_A=\emptyset$;

\medskip$(5)$\ $u \prec x \prec y \prec v$, and $f\big([u,x)_A
\cup (y, v]_A\big) \subset \stackrel{\circ}D_0$ or $f\big([u,x)_A \cup
(y, v]_A\big) \subset \mathbb{R}^2-D_0$,  where
$$
D_0= \cup\{V: V \mbox{is a bounded connected component of}\ \mathbb R^2-f\big([x,y]_A\big)-[u, v]_A\}
\vspace{-2mm}
$$
$$
\hspace{-76mm}\cup f\big([x,y]_A\big)\cup [u, v]_A.
$$
\end{thm}

\begin{proof} Let $K=f\big([x,y]_A\big)$, $K_0=f\big((x,y)_A\big)$, and $P=K\cup (u, v)_A$. Then
(\ref{eq:3-2-1}) can be written as
\begin{equation}\label{eq:3-2-2}
K \cap (u, v)_A=\emptyset, \ \  f^{-1}(K_0)=(x,y)_A,
\end{equation}
and $\mathbb R^2-f\big([x,y]_A\big)-[u, v]_A=\mathbb R^2-P$.
Note that $\{u, v\}\subset K$, $K_0=K-\{u, v\}$, $K_0$ is arcwise connected, and $K$ is a Peano continuum (see \cite[Theorem 8.18]{Na92}).

\medskip {\textbf{Claim 1.}\ {\it $\overline{W}=W \cup K \cup (u,v)_A \subset X$ and there is a
connected open neighborhood $V_0$ of $(u, v)_A$ in $X$ such that  $V_0-(u,v)_A \subset \overline{V_0}-[u,v]_A \subset W$.}

\medskip{\textbf{Proof of Claim 1.}}\
Since $W$ is a bounded connected
component of $\mathbb{R}^2-P$, and since $K \cup A \subset
X$, we have
\begin{equation}\label{eq:3-3} W \subset~ \stackrel{\circ}X \mbox{\ \ and\ \ } \overline{W}
\subset W \cup K \cup (u,v)_A \subset X.
\end{equation}
 Since $f\big((x,y)_A\big)=K_0 \subset X-[u,v]_A$,
there is a connected neighborhood $U_0$ of $(x,y)_A$ in $X$ such
that $f(U_0) \subset \mathbb{R}^2-[u,v]_A$, which with
$f^{-1}(K_0)=(x,y)_A$ implies that
\begin{equation}\label{eq:3-4}
 f\big(U_0-(x,y)_A\big) \subset \mathbb{R}^2-[u,v]_A-K_0
=\mathbb{R}^2-P.
\end{equation}
 Since
$U_0-(x,y)_A$ is connected, $U_0 \cap U_q-(x,y)_A \supset U_0 \cap
U_q-A \neq \emptyset$, and $f(U_q-A) \subset W$, from (\ref{eq:3-4}) we get
\begin{equation} \label{eq:3-5}
f\big(U_0-(x,y)_A\big) \subset W,
\end{equation} which with (\ref{eq:3-3}) implies that
\begin{equation}\label{eq:3-6}
\begin{split}
 X &\supset W \cup K \cup (u,v)_A \supset \overline{W}
\supset  W \cup f\left(\overline{U_0-(x,y)_A}\right)\\
&= W \cup f(\overline {U_0}) \supset W \cup f\big([x,y]_A\big)=W \cup K.
\end{split}
\end{equation}
If $u \neq v$, then $(u, v)_A \neq \emptyset$, and from $K \cap (u,
v)_A=\emptyset$ we see that there is a connect open neighborhood
$V_0$ of $(u, v)_A$ in $X$ such that $V_0 \cap K =\emptyset$,
$\overline{V_0} \cap K=\{u, v\}$, and $\overline{V_0}$ is a disc. It
follows from $V_0 \cap K=\emptyset$ that
\begin{equation} \label{eq:3-7}
V_0-(u,v)_A \subset W \mbox{\ \ or\ \ } \big(V_0-(u,v)_A\big) \cap W=\emptyset.
\end{equation}

If $u \neq v$ and $(u, v)_A \nsubseteq \overline{W}$, then
$V_0-(u,v)_A \nsubseteq W$ and from (\ref{eq:3-7}) we get
$\big(V_0-(u,v)_A\big)\cap W=\emptyset$, which induces that $(u,v)_A \cap
\overline{W}=\emptyset$ and from (\ref{eq:3-6}) we get $\overline{W}=W \cup
K$. By Lemma \ref{arc}, there is an arc $J$ such that
$\stackrel{~\circ}J \subset W$ and $\partial J=\{u, v\}$. Let $E$ be
the disc in $\mathbb{R}^2$ such that $\partial E=J \cup [u,v]_A$.
Then both $\stackrel{~\circ}E \cap W \neq \emptyset$ and
$(\mathbb{R}^2-E) \cap W \neq \emptyset$, and hence both
$\stackrel{~\circ}E \cap K \neq \emptyset$ and $(\mathbb{R}^2-E)
\cap K \neq \emptyset$. Thus $K_0=K-\{u, v\}$ is not connected. But
this contradicts that $K_0$ is a connected set. Therefore, if $u
\neq v$ then $(u, v)_A \subset \overline{W}$, which with (\ref{eq:3-6})
implies
\begin{equation}\label{eq:3-8}
\overline{W}=W \cup K \cup (u,v)_A \subset X.
\end{equation}
 From (\ref{eq:3-8}) we get $\big(V_0-(u,v)_A\big) \cap W \neq
\emptyset$, which with (\ref{eq:3-7}) and $\overline{V_0}\cap K=\{u, v\}$
induces
\begin{equation}\label{eq:3-9}
 V_0-(u,v)_A \subset \overline{V_0}-[u,v]_A \subset W.
\end{equation}
 Note that if $u=v$ then $(u,v)_A=\emptyset$ and we
can take $V_0=\emptyset$. Thus (\ref{eq:3-9}) is still true for the case
$u=v$. From (\ref{eq:3-6}) we see that (\ref{eq:3-8}) is also true for the case $u=v$.}
\end{proof}

\medskip {\textbf{Claim 2.}\ {\it {\rm (C.2.1)}\ There is a
disc $D_0$ in $\mathbb{R}^2$ such that $$[u,v]_A \subset \partial
D_0 \subset K \cup (u,v)_A \subset Bd(W) \subset \overline{W}
\subset D_0 \subset X.$$

{\rm (C.2.2)} The unique unbounded connected component of
$\mathbb{R}^2-P$ is $\mathbb{R}^2-D_0$, and every bounded
connected component of $\mathbb{R}^2-P$ is contained in
$\stackrel{~\circ}D_0$.}

\medskip{\textbf{Proof of Claim 2.}}\ It suffices to show
(C.2.1), since (C.2.1) obviously implies (C.2.2). By (\ref{eq:3-8}) we have
$$K \cup (u,v)_A=K \cup [u,v]_A \subset Bd(W) \subset \overline{W}.$$
Thus it suffices to show that there is a disc $D_0$ in
$\mathbb{R}^2$ satisfying $[u,v]_A \subset \partial D_0 \subset K
\cup (u,v)_A$ and $\overline{W} \subset D_0 \subset X$.

\medskip{\textbf{Case 1.}}\ If $u \neq v$, then there is an arc
$L \subset K$ such that $\partial L=\{u, v\}$. Write $C_0=L \cup (u,
v)_A$. Then $C_0$ is a circle and $C_0 \cup K \cup (u,v)_A \subset
X$. Let $D_0=D(C_0)$ be the disc in $\mathbb{R}^2$ such that
$\partial D_0=C_0$. Then $D_0 \subset X$. Let the neighborhood $V_0$
of $(u,v)_A$ in $X$ be as above. By (\ref{eq:3-9}) we have $W \cap
\stackrel{~\circ}D_0 \supset V_0 \cap \stackrel{~\circ}D_0 \neq
\emptyset$, which with $W \cap \partial D_0 \subset W \cap \big(K
\cup (u,v)_A\big)=\emptyset$ implies that $W \subset
\stackrel{~\circ}D_0$ and hence $\overline{W} \subset D_0$.

\medskip{\textbf{Case 2.}}\ If $u=v$, then by (\ref{eq:3-2}) we have $u
\in (x,y)_A$. Take a point $w_0 \in K_0$. By Lemma \ref{arc}, there is
an arc $J$ such that $\stackrel{~\circ}J \subset W$ and $\partial
J=\{u, w_0\}$. Write
$$G_0=[-1,0] \times [0,1] \mbox{\ \ and\ \ } G_1=[0,1] \times [0,1].$$
Take a homeomorphism $h: \mathbb{R}^2 \rightarrow \mathbb{R}^2$ such
that $h(u)=(0,0)$,
$$h \big([x,y]_A\big)=[-1,1] \times \{0\}, \ \ h(J)=\{0\} \times [0,2],$$
and $h(X) \supset [-2,2] \times [0,2]$. Take $x' \in (x,u)_A$ and
$y' \in (u,y)_A$ such that $h \circ f\big([x,x']_A \cup
[y',y]_A\big) \subset G_0 \cup G_1$. If
$$h \circ f\big([x,x']_A \cup [y', y]_A\big) \subset G_0
\mbox{\ \ or \ \ } h \circ f\big([x,x']_A \cup [y', y]_A\big)
\subset G_1,$$ then $h(\stackrel{~\circ}J)$ is contained in the
unbounded connected component of $\mathbb{R}^2-h \circ
f\big([x,y]_A\big)$, and hence $\stackrel{~\circ}J$ is contained in
the unbounded connected component of $\mathbb{R}^2-f
\big([x,y]_A\big)=\mathbb{R}^2-K=\mathbb{R}^2-K-[u,v]_A$. But this
contradicts with $\stackrel{~\circ}J \subset W$. Thus we have
\begin{equation}\label{eq:3-10}
 h \circ f \big([x,x']_A\big) \subset G_0 \mbox{\ \ and\ \ }
h \circ f \big([y',y]_A\big) \subset G_1.
\end{equation}
 or
\begin{equation}\label{eq:3-11}
 h \circ f \big([x,x']_A\big) \subset G_1 \mbox{\ \ and\ \ }
h \circ f \big([y',y]_A\big) \subset G_0.
\end{equation}
Take arcs $L_1, L_2$ and $L_3$ such that
$$L_1 \subset h \circ f\big([x,x']_A\big), \ \ L_2 \subset h \circ f\big([x',y']_A\big), \ \
L_3 \subset h \circ f \big([y', y]_A\big)$$ and
$$\partial L_1=\{h \circ f(x), h \circ h(x')\}, \ \ \partial L_2=\{h \circ f(x'), h \circ f(y')\}, \ \
\partial L_3=\{h \circ f(y'), h \circ f(y)\}.$$
Take arcs $L_4 \subset L_1$, $L_5 \subset L_2$ and $L_6 \subset L_3$
such that $$h \circ f(x)=h(u)=(0,0) \in \partial L_4,\ h \circ
f(y)=h(v)=(0,0) \in \partial L_6,$$ $L_4 \cap L_5=\partial L_4 \cap
\partial L_5 \neq \emptyset$, and $L_5 \cap L_6=\partial L_5 \cap \partial L_6 \neq
\emptyset$. Let $Q=L_4 \cup L_5 \cup L_6$. Then $Q$ is a circle, $Q
\subset L_1 \cup L_2 \cup L_3 \subset h \circ
f\big([x,y]_A\big)=h(K)$, and $h(\stackrel{~\circ}J)=\{0\} \times
(0,2)$ is contained in the interior of the disc $D(Q)$. Let
$C_0=h^{-1}(Q)$, and $D_0=h^{-1}(D(Q))$. Then $[u,v]_A=\{u\} \subset
\partial D_0=C_0 \subset K=K \cup [u,v]_A \subset X$, $D_0 \subset
X$, and $\stackrel{~\circ}J \subset W \cap \stackrel{~\circ}D_0$,
which with $W \cap \partial D_0=\emptyset$ implies $W \subset
\stackrel{~\circ}D_0$ and hence $\overline{W} \subset D_0$.

\medskip Therefore, no matter whether $u \neq v$ or $u=v$,
(C.2.1) is true, and Claim 2 is proved. \hfill$\Box$

\medskip {\textbf{Claim 3.}\ {\it Let $\Omega=\partial
D_0-[u,v]_A$. Let $V$ be a bounded connected component of
$\mathbb{R}^2-P$ but $V \neq W$. Then

\medskip {\rm (C.3.1)}\ $\Omega$ is an open arc contained in
$K_0$.

\medskip{\rm (C.3.2)} $Bd(V) \cap \Omega$ contains at most one
point.

\medskip{\rm (C.3.3)} $Bd(V) \cap [u,v]_A=\emptyset$.

\medskip{\rm (C.3.4)} $Bd(V) \subset K_0$.

\medskip{\rm (C.3.5)} $f(X) \cap V=\emptyset$.}

\medskip{\textbf{Proof of Claim 3.}}\ (1)\ Since $\partial
D_0 \subset K \cup (u,v)_A=K_0 \cup [u,v]_A$, (C.3.1) is clear.

\medskip (2)\ If the conclusion (C.3.2) is not true, then $Bd(V)
\cap \Omega$ contains two different points $w$ and $z$. Take an
imbedding $\lambda: [0,3] \rightarrow \Omega$ such that
$\lambda(0)=\omega$ and $\lambda(2)=z$. Write $w'=\lambda(1)$ and
$z'=\lambda(3)$. Then $\{w', z'\} \subset \Omega \subset K_0 \subset
\partial D_0 \subset Bd(W)$. By Lemma \ref{arc} and Claim 2, there exist
arcs $\Gamma \subset \overline{V} \subset D_0$ and $\Lambda \subset
\overline{W} \subset D_0$ such that
$$\partial \Gamma=\{w,z\}, \ \ \stackrel{\circ}\Gamma \subset V, \ \
\partial \Lambda=\{w', z'\}, \mbox{\ and\ } \stackrel{\circ}\Lambda \subset W,$$
which implies that $\Gamma \cap \Lambda=\stackrel{\circ}\Gamma \cap
\stackrel{\circ}\Lambda \subset W \cap V=\emptyset$. On the other
hand, let $\prec$ be the order on $\lambda\big([0,3]\big)$ induced
by $\lambda$. Since $w \prec w' \prec z \prec z'$ and $\Gamma \cup
\Lambda \subset D_0$, we must have $\Gamma \cap \Lambda \neq
\emptyset$. These lead to a contradiction. Thus (C.3.2) is true.

\medskip (3)\ Let the open neighborhood $V_0$ of $(u,v)_A$ in
$X$ be as above. From (\ref{eq:3-9}) we get
\begin{equation}\label{eq:3-12}
\big(V_0-(u,v)_A\big) \cap V \subset W \cap V=\emptyset.
\end{equation}
Since $V$ is open in $\mathbb{R}^2$, from (\ref{eq:3-12}) we
get
\begin{equation}\label{eq:3-13}
V_0 \cap V \subset \overline{V_0-(u,v)_A} \cap
V=\emptyset.
\end{equation}
 Since $V_0$ is open in
$X$ and $\overline{V} \subset X$, from (\ref{eq:3-13}) we get
\begin{equation}\label{eq:3-14}
(u,v)_A \cap Bd(V) \subset V_0 \cap \overline{V}=\emptyset.
\end{equation}
Note that $V$ is a connected component of
$\mathbb{R}^2-K-(u,v)_A$. Thus $Bd(V) \subset K \cup (u,v)_A$, which
with (\ref{eq:3-14}) induces $Bd(V) \subset K$. Hence, for proving (C.3.3),
it suffices to show $\{u,v\} \cap Bd(V)=\emptyset$, and further, by
symmetry, it suffices to show $u \notin Bd(V)$.

\medskip Assume on the contrary that $u \in Bd(V)$. Take a point
$z_0 \in Bd(V)-\{u\}$. By Lemma \ref{arc}, there is an arc $\Gamma$ with
$\partial \Gamma=\{u, z_0\}$ and $\stackrel{\circ}\Gamma \subset V$.
Write $a=|z_0-u|$. Choose $b \in (0, a/2]$ such that
$|f(w)-f(z)|<a/2$ for any $w, z \in X$ with $|w-z| \leq 2b$.

\medskip{\textbf{Case 1.}}\ If $u \neq v$, take $x_0 \in
(x,y)_A$ such that $[x, x_0]_A \subset B(x,b)$. Since $f
\left([x_0, y]_A\right)$ is a compact set in $K-\{u\}$, there exists
$c \in (0, b]$ such that $B(u,c) \cap f \left([x_0,
y]_A\right)=\emptyset$, which with $Bd(V)-\{u,v\} \subset K_0=f
\left((x,y)_A\right)$ implies that
\begin{equation}\label{eq:3-15}
\begin{split}
 Bd(V) \cap B(u,c)-\{u\} &= Bd(V) \cap B(u,c)-\{u,v\} \\
&= B(u,c) \cap K_0 \subset f \big((x,x_0)_A\big).
\end{split}
\end{equation}

{\textbf{Case 2.}}\ If $u=v$, then we take a point $w_0 \in
\Omega-Bd(V)$ and take an arc $J$ such that $\stackrel{~\circ}J
\subset W$ and $\partial J=\{u, w_0\}$. Let the squares $G_0$ and
$G_1$, the homeomorphism $h: \mathbb{R}^2 \rightarrow \mathbb{R}^2$,
and the points $x' \in (x,u)_A$ and $y' \in (u, y)_A$ be the same as
in the proof of Claim 2. By symmetry we may assume that (\ref{eq:3-11})
holds. By Lemma \ref{arc}, there is an arc $J'$ in the sphere
$\mathbb{R}^2 \cup \{\infty\}$ such that $\partial J'=\{w_0,
\infty\}$ and $J' \cap K=J' \cap C_0=\{w_0\}$. This leads to $Bd(V)
\cap \big(J' \cup J-\{u\}\big) \subset \overline{V }\cap (J' \cup
W)=\emptyset$. Thus there exists $r>0$ such that
$$Bd(V) \cap B(u,r) \subset h^{-1}(G_0) \mbox{\ \ or \ \ } Bd(V) \cap B(u,r) \subset h^{-1}(G_1).$$
By symmetry, we may assume
\begin{equation} \label{eq:3-16}
Bd(V) \cap B(u,r) \subset h^{-1}(G_1).
\end{equation}
Choose $x_0 \in (x,x')_A$ such that $[x,x_0]_A
\subset B(x,b)$. Similar to Case 1, there exists $c \in \big(0,
\min\{r, b\}\big]$ such that
$$B(u,c) \cap f \big([x_0, y']_A\big)=\emptyset,$$
which with $f \big((x,y)_A\big)=K_0$ implies that
\begin{equation}\label{eq:3-17}
 B(u,c) \cap K_0 \subset f \big((x,x_0)_A\big) \cup f \big((y', y)_A\big).
\end{equation}
 Noting $Bd(V)-\{u\} \subset K_0$, from (\ref{eq:3-17}),
(\ref{eq:3-11}) and (\ref{eq:3-16}) we see that (\ref{eq:3-15}) also holds in Case 2.

\medskip Clearly, no matter whether in Case 1 or in Case 2,
there exist points $u_0$ and $v_0$ in $Bd(V) \cap B(u,c)-\{u\}
\subset B(u,c) \cap f \big((x,x_0)_A\big)$ such that the following
property holds:

\medskip (P.1)\ For any arc $L$ with $\partial L=\{u_0, v_0\}$
and $\stackrel{~\circ} L \subset V$, it holds that $L \cap
\stackrel{\circ}\Gamma \neq \emptyset$, that is, $u_0$ and $v_0$ lie
on different sides of $\Gamma$.

\medskip From (P.1) we get

\medskip (P.2)\ If $L$ is an arc in $K_0$ with $\partial
L=\{u_0, v_0\}$, then $L \cap \stackrel{\circ}\Gamma=\emptyset$, and
hence $\mathrm{diam}(L) \geq |z_0-u_0| \geq a-c \geq a/2$.

\medskip On the other hand, it follows from (\ref{eq:3-15}) that there
exist points $x_1$ and $x_2$ in $(x, x_0)_A$ such that $f(x_1)=u_0$
and $f(x_2)=v_0$. Hence there is an arc $L \subset f \big([x_1,
x_2]_A\big) \subset K_0$ such that $\partial L=\{u_0, v_0\}$. Since
$[x_1, x_2]_A \subset (x, x_0]_A \subset B(x,b)$, we have
$\mathrm{diam}(L) \leq \mathrm{diam}\Big(f\big([x_1,
x_2]_A\big)\Big) \leq \mathrm{diam}\Big(f\big(B(x,b)\big)\Big)<
a/2$. But this contradicts (P.2). Thus we must have $u \notin
Bd(V)$, and hence the conclusion (C.3.3) is true.

\medskip (4) Since $Bd(V) \subset K \cup (u,v)_A=K_0 \cup
[u,v]_A$, conclusion (C.3.4) is a direct corollary of (C.3.3).

\medskip (5) If $f(X) \cap V \neq \emptyset$, take an arc $L$
with $\partial L=\{w, z\}$ such that $f(w) \in W$, $f(z) \in V$, and
$L \subset X-(x,y)_A$. Then by (C.3.4) we have $f(L) \cap K_0
\supset f(L) \cap Bd(V) \neq \emptyset$. On the other hand, since
$f^{-1}(K_0)=(x,y)_A$ and $L \cap (x,y)_A=\emptyset$, we have $f(L)
\cap K_0=\emptyset$. These lead to a contradiction. Thus conclusion
(C.3.5) holds, and Claim 3 is proved. \hfill$\Box$

\medskip {\textbf{Claim 4.}\ \ {\it {\rm (C.4.1)}\ For each
$z \in K_0$, there exists an open neighborhood $U_z$ of $z$ in
$\mathbb{R}^2$ such that $U_z \cap [u,v]_A=\emptyset$ and $U_z \cap
f(X) \subset W \cup K_0$.

\medskip {\rm (C.4.2)}\ Let the open arc $\Omega=\partial
D_0-[u,v]_A \subset K_0$ be as in Claim 3, let
$U_{\Omega}=\bigcup\{U_z: z \in \Omega\}$. Then $U_{\Omega}$ is an
open neighborhood of $\Omega$ in $\mathbb{R}^2$, $U_{\Omega} \cap
[u,v]_A=\emptyset$, and $$U_{\Omega} \cap f(X) \subset W \cup K_0
\subset \stackrel{~\circ}D_0 \cup K_0=D_0-[u,v]_A.$$}

{\textbf{Proof of Claim 4.}}\ Evidently, (C.4.1)
implies (C.4.2). Thus it suffices to show (C.4.1). Let
$Y_z=f^{-1}(z)$. Then $Y_z$ is a nonempty compact set and $Y_z
\subset f^{-1}(K_0)=(x,y)_A$. If (C.4.1) does not hold, then there
exist points $w_1, w_2, w_3, \cdots$ in $X$ and $w \in Y_z$ such
that $\lim_{n \rightarrow \infty} f(w_n)=z$,
\begin{equation} \label{eq:3-18}
f\big(\{w_1, w_2, \cdots\}\big) \subset \mathbb{R}^2-W-K_0.
\end{equation}
 and $\lim_{n \rightarrow \infty} w_n=w$. Let the
neighborhood $U_0$ of $(x,y)_A$ in $X$ be the same as in (\ref{eq:3-5}). Then
there exists $n \in \mathbb{N}$ such that $w_n \in U_0-(x,y)_A$, and
from (\ref{eq:3-5}) we obtain $f(w_n) \in f(U_0-(x,y)_A) \subset W$. But this
contradicts with (\ref{eq:3-18}). Thus (C.1) holds, and Claim 4 is
proved. \hfill$\Box$

\medskip Let $D_0, \Omega$ and $U_{\Omega}$ be as above. Write
\begin{equation}\label{eq:3-19} V_{\Omega}=U_{\Omega}-D_0,\ \ Y_0=\mathbb{R}^2-U_{\Omega}-\stackrel{~\circ}D_0,
\ \ Y=Y_0 \cup D_0.
\end{equation}
 Then $V_{\Omega}$
is an open set in $\mathbb{R}^2-D_0$, $U_{\Omega}=(U_{\Omega} \cap
\stackrel{~\circ}D_0) \cup \Omega \cup V_{\Omega}$, $Y_0$ and $Y$
are closed in $\mathbb{R}^2$,
$Y_0=\mathbb{R}^2-V_{\Omega}-\Omega-\stackrel{~\circ}D_0=(\mathbb{R}^2-V_{\Omega}-D_0)\cup
[u,v]_A$, and $Y=\mathbb{R}^2-V_{\Omega}$. Write
\begin{equation}\label{eq:3-20}
\begin{split}
\mathbb{V}&= \bigcup \big\{ V: V \mbox{\ is a bounded
connected
component}\\
& \ \ \ \ \ \ \ \ \ \ \ \mbox{\ of\ } \mathbb{R}^2-K-[u,v]_A \mbox{\ but\ } V
\neq W \big\}.
\end{split}
\end{equation}
Then $\mathbb{V}$ is an open set in $\stackrel{~\circ}D_0$. By Claim
2.1.2, we have
\begin{equation}\label{eq:3-21}
 D_0=K_0 \cup [u,v]_A \cup W \cup \mathbb{V}.
\end{equation}
 By Claim 4, we get $f(X) \cap V_{\Omega}=f(X)
\cap U_{\Omega}-D_0=\emptyset$, which with the conclusion (C.3.5) of
Claim 3 implies
\begin{equation} \label{eq:3-22}
f(X) \subset \mathbb{R}^2-V_{\Omega}-\mathbb{V}=Y-V_{\Omega} \subset Y.
\end{equation}

\medskip {\textbf{Claim 5.}\ \ {\it If $\{u,v\} \subset
[x,y]_A$, then
$$\mathrm{Fix}(f) \cap D_0=\mathrm{Fix}(f) \cap W \neq \emptyset.$$}

\vspace{-5mm}{\textbf{Proof of Claim 5.}}\ Let $Y_0$ and $Y$ be
as above. Since $[u,v]_A$ is an absolute retract, there exists a
retraction $\gamma: Y_0 \rightarrow [u,v]_A$. Since $Y_0 \cap
D_0=[u,v]_A$ and both $Y_0$ and $D_0$ are closed subsets of $Y$, the
retraction $\gamma$ can be extended to a retraction $\xi: Y
\rightarrow D_0$ by $\xi|Y_0=\gamma$ and $\xi|D_0=\mathrm{id}$. By
(\ref{eq:3-22}), we can define a continuous map $\varphi: D_0 \rightarrow
D_0$ by $\varphi=\xi \circ f|D_0$.
By Brouwer fixed point theorem, we have
$\mathrm{Fix}(\varphi) \neq \emptyset$.

\medskip From (\ref{eq:3-22}) we get $\varphi(\mathbb{V})=\xi \circ
f(\mathbb{V}) \subset \xi(Y-\mathbb{V})=D_0-\mathbb{V}$. Thus
$\mathrm{Fix}(\varphi) \cap \mathbb{V}=\emptyset$. Since
$\varphi^{-1}(K_0)=f^{-1}\circ \xi^{-1}(K_0)=f^{-1}(K_0)=(x,y)_A
\subset D_0-K_0$, we have $\mathrm{Fix}(\varphi) \cap
K_0=\emptyset$.

\medskip For any $w \in [u,v]_A$, if $w \in (x,y)_A$ then
$\varphi(w)=\xi \circ f(w) \in \xi \circ
f\big((x,y)_A\big)=\xi(K_0)=K_0 \subset D_0-[u,v]_A$; if $w=x$ then
$\varphi(w)=\xi \circ f(x)=\xi(u)=u \neq x$; if $w=y$ then we also
have $\varphi(w)=\xi \circ f(y)=\xi(v)=v \neq y$. Thus
$\mathrm{Fix}(\varphi) \cap [u,v]_A =\emptyset$.

\medskip To sum up, we get $\mathrm{Fix}(\varphi) \cap
\big(\mathbb{V} \cup K_0 \cup [u,v]_A\big)=\emptyset$, which with
(\ref{eq:3-21}) implies $\mathrm{Fix}(\varphi) \cap W=\mathrm{Fix}(\varphi)
\neq \emptyset$. For any $w \in W$, if $f(w) \in Y-D_0$ then
$\varphi(w)=\xi \circ f(w) \in \xi(Y-D_0) \subset \xi (Y_0)=[u,v]_A
\subset D_0-\{w\}$. Thus if $w \in \mathrm{Fix}(\varphi)$ then $f(w)
\in D_0$ and hence $f(w)=\xi \circ f(w)=\varphi(w)=w \in
\mathrm{Fix}(f)$. Conversely, if $w \in \mathrm{Fix}(f) \cap D_0$,
then $\varphi(w)=\xi \circ f(w)=\xi(w)=w \in \mathrm{Fix}(\varphi)$.
Hence we have $\mathrm{Fix}(f) \cap
D_0=\mathrm{Fix}(\varphi)=\mathrm{Fix}(\varphi) \cap
W=\mathrm{Fix}(f) \cap D_0 \cap W=\mathrm{Fix}(f) \cap W \neq
\emptyset$. Claim 5 is proved. \hfill$\Box$

\medskip In Claim 5 we have discussed the case that $[u,v]_A
\subset [x,y]_A$. In the following we will consider the case
$[u,v]_A \nsubseteq [x,y]_A$. Recall that $\prec$ is the order on
$A$ induced by the homeomorphism $\lambda: A \rightarrow [0,1]$.
Since $x \prec y$, from (\ref{eq:3-2}) we see that $[u,v]_A \nsubseteq
[x,y]_A$ if and only if one of the following conditions holds:

\medskip (2)\ $x \preceq v \prec y \prec u$;

\medskip (3)\ $v \prec x \prec u \preceq y$;

\medskip (4)\ $v \prec x \prec y \prec u$;

\medskip (5)\ $u \prec x \prec y \prec v$.

\medskip As mentioned at the beginning  of Section 2, we regard any point $t \in \mathbb{R}$
and the point $(t,0) \in \mathbb{R}^2$ as the same, that is, we
regard $\mathbb{R}$ as a subspace of $\mathbb{R}^2$. For $r,s \in
\mathbb{R}$ with $r<s$, we write $[s,r]=[r,s]$. Hence, for any $w, z
\in \mathbb{R}$, we have both $[w,z] \subset \mathbb{R}$ and
$[w,z]=[w,z] \times \{0\} \subset \mathbb{R}^2$. We also have both
$[w,z] \times [r,s] \subset \mathbb{R}^2$ and $[w,z] \times
[r,s]=[w,z] \times [r,s] \times \{0\} \subset \mathbb{R}^3$, but we
cannot have $[w,z] \times [r,s]=[w,z] \times \{0\} \times [r,s]
\subset \mathbb{R}^3$.

\medskip We can take a homeomorphism $h: \mathbb{R}^2
\rightarrow \mathbb{R}^2$ such that

\medskip (h.1)\ $h(X)=[0,8] \times [0,4]$, $h(A) \subset [1,7]
\big(=[1,7] \times \{0\}\big)$;

\medskip (h.2)\ $h \big([x,y]_A\big) \subset [3,5]$ with
$h(x)<h(y)$.

\medskip Write $X'=h(X)$ and $W'=h(W)$. Let $f'=h \circ f \circ
h^{-1}|X': X' \rightarrow \mathbb{R}^2$. Then $\mathrm{Fix}(f')=h
\big(\mathrm{Fix}(f)\big)$ and $\mathrm{Fix}(f') \cap
W'=h\big(\mathrm{Fix}(f) \cap W\big)$. Therefore, for the
convenience of statement, in the following we may directly assume
that

\medskip (H.1)\ $X=[0,8] \times [0,4]$, and $A \subset [1,7]$;

\medskip (H.2)\ $[x,y]_A \subset [3,5]$ with $x<y$.

\medskip\noindent  In addition, for any $w, z \in A$, since
we assume $A \subset [1,7] \subset \mathbb{R}$, in the following we
can write $[w,z]$ for $[w,z]_A$.

\medskip {\textbf{Claim 6.}\ \ {\it If $x \leq v<y<u$, and
$f\big([y,u]\big) \cap [y,u]=\emptyset$, then $\mathrm{Fix}(f) \cap
D_0=\mathrm{Fix}(f) \cap W \neq \emptyset$.}

\medskip{\textbf{Proof of Claim 6.}}\ For $n=1,2$, define
the projection $p_n: \mathbb{R}^2 \rightarrow \mathbb{R}$ by
$p_n(z)=r_n$ for any $z=(r_1, r_2) \in \mathbb{R}^2$. Since
$f\big([y,u]\big) \cap [y,u]=\emptyset$, there exists $a \in \big(0,
(y-v)/2\big]$ such that
\begin{equation} \label{eq:3-23}
f\big([y-a, u+a] \times [0,a]\big) \cap \big([y-a,u+a] \times [0,a]\big)=\emptyset.
\end{equation}
Write $E=[y-a, u+a] \times [0,a]$. Take a
homeomorphism $\psi: X \rightarrow X$ satisfying the following three
conditions:
\begin{equation}\label{eq:3-24}
 \psi|(X-E)=\mathrm{id} \mbox{\ \ and\ \ } \psi(E)=E,
 \end{equation}
\begin{equation}\label{eq:3-25}
 p_1(\psi(z)) \leq p_1(z) \mbox{\ \ and\ \ } p_2(\psi(z))=p_2(z) \mbox{\ \ for any\ \ } z \in E,
\end{equation}
\begin{equation} \label{eq:3-26}
\psi([x,u])=[x,y] \mbox{\ \ with\ \ } \psi(x)=x \mbox{\ \ and\ \ } \psi(u)=y.
\end{equation}
 Let $F=f \circ \psi: X \rightarrow \mathbb{R}^2$. By
(\ref{eq:3-24}) we have
\begin{equation} \label{eq:3-27} F|(X-E)=f|(X-E) \mbox{\ \ and\ \ } F(E)=f(E),
 \end{equation} which implies
\begin{equation} \label{eq:3-28} \mathrm{Fix}(F)-E=\mathrm{Fix}(f)-E.
\end{equation} From (\ref{eq:3-26}) we obtain
\begin{equation} \label{eq:3-29} F\big([x,u]\big)=f\big([x,y]\big)=
K \mbox{\ \ with\ \ } F(x)=u \mbox{\ \ and\ \ } F(u)=v,
\end{equation}
\begin{equation} \label{eq:3-30} F^{-1}(K_0)=\psi^{-1} \circ f^{-1}(K_0)=\psi^{-1}\big((x,y)\big)=(x,u).
\end{equation}
 Let the neighborhood $U_0$ of $(x,y)_A$ in $X$ be
the same as in (2.5). Write $U_{01}=\psi^{-1}(U_0)$. Then $U_{01}$
is a neighborhood of $(x,u)$ in $X$, and by (2.26) and (2.5) we get
$$F\big(U_{01}-(x,u)\big)=f\big(U_0, (x,y)\big) \subset W.$$
Hence, from Claim 5 we obtain
\begin{equation}\label{eq:3-31}
 \mathrm{Fix}(F)\cap D_0=\mathrm{Fix}(f) \cap W \neq \emptyset.
 \end{equation}
  From (\ref{eq:3-23}) and (\ref{eq:3-27}) we get
\begin{equation} \label{eq:3-32}
\mathrm{Fix}(f) \cap E =\emptyset,
\end{equation}
\begin{equation} \label{eq:3-33} F(E) \cap E=f(E) \cap E=\emptyset.
\end{equation}
 We note that (\ref{eq:3-33}) implies
\begin{equation} \label{eq:3-34} \mathrm{Fix}(F) \cap E =\emptyset.
 \end{equation}
 By (\ref{eq:3-32}), (\ref{eq:3-28}) and (\ref{eq:3-34}) we obtain
$$\mathrm{Fix}(f)=\mathrm{Fix}(f)-E=\mathrm{Fix}(F)-E=\mathrm{Fix}(F),$$
which with (\ref{eq:3-31}) implies $\mathrm{Fix}(f) \cap D_0=\mathrm{Fix}(f)
\cap W \neq \emptyset$. Claim 6 is proved. \hfill$\Box$

\medskip By symmetry, from Claim 6 we get

\medskip {\textbf{Claim 7.}\ \ {\it If $v < x < u \leq y$,
and $f\big([v,x]\big) \cap [v,x]=\emptyset$, then $\mathrm{Fix}(f)
\cap D_0=\mathrm{Fix}(f) \cap W \neq \emptyset$.}

\medskip {\textbf{Claim 8.}\ \ {\it If $v<x<y<u$, and
$f\big([v,x]\big) \cap [v,x]=f\big([y,u]\big) \cap [y,u]=\emptyset$,
then $\mathrm{Fix}(f) \cap D_0=\mathrm{Fix}(f) \cap W \neq
\emptyset$.}

\medskip{\textbf{Proof of Claim 8.}}\ Let the homeomorphism
$\psi: X \rightarrow X$ be the same as in the proof of Claim 6,
and let $F=f \circ \psi: X \rightarrow \mathbb{R}^2$. From the proof
of Claim 6 we see that the equality
$\mathrm{Fix}(F)=\mathrm{Fix}(f)$ still holds, although the
condition $x \leq v<y$ is replaced by $v<x<y$. From Claim 7 we
get $\mathrm{Fix}(F) \cap D_0=\mathrm{Fix}(F) \cap W \neq
\emptyset$. Hence $\mathrm{Fix}(f) \cap D_0=\mathrm{Fix}(f) \cap W
\neq \emptyset$. Claim 8 is proved. \hfill$\Box$

\medskip {\textbf{Claim 9.}\ \ {\it If $u<x<y<v$, and
$f\big([u,x) \cup (y,v]\big) \subset \stackrel{~\circ}D_0$ or
$f\big([u,x) \cup (y,v]\big) \subset \mathbb{R}^2-D_0$, then
$\mathrm{Fix}(f) \cap D_0=\mathrm{Fix}(f) \cap W \neq \emptyset$.}

\medskip{\textbf{Proof of Claim 9.}}\ Noting that
$f^{-1}(K_0)=(x,y)$, from (\ref{eq:3-21}) and (\ref{eq:3-22}) we see that

\medskip (i) $f\big([u,x) \cup (y,v]\big) \subset
\stackrel{~\circ}D_0$ if and only if $f\big([u,x) \cup (y,v]\big)
\subset W$;

\medskip (ii) $f\big([u,x) \cup (y,v]\big) \subset
\mathbb{R}^2-D_0$ if and only if $f\big([u,x) \cup (y,v]\big)
\subset \mathbb{R}^2-W-K-[u,v]$.

\medskip We now begin the proof of the claim. Let $Y_0$ and $Y$
be the same as in (\ref{eq:3-19}).

\medskip If $f\big([u,x) \cup (y,v]\big) \subset
\stackrel{~\circ}D_0$, then we take a retraction $\gamma: Y_0
\rightarrow [u,v]_A$ and define $\xi: Y \rightarrow D_0$ by
$\xi|Y_0=\gamma$ and $\xi|D_0=\mathrm{id}$. Let $\varphi=\xi \circ
f|D_0: D_0 \rightarrow D_0$. Similar to the proof of Claim 5, we
have $\mathrm{Fix}(f) \cap
D_0=\mathrm{Fix}(\varphi)=\mathrm{Fix}(\varphi) \cap
W=\mathrm{Fix}(f) \cap D_0 \cap W=\mathrm{Fix}(f) \cap W \neq
\emptyset$.

\medskip In the following we consider the case $f\big([u,x) \cup
(y,v]\big) \subset \mathbb{R}^2-D_0$. Since we attend to
$\mathrm{Fix}(f) \cap D_0$ only, we can replace $f$ by $f|D_0: D_0
\rightarrow \mathbb{R}^2$ if $X-D_0 \neq \emptyset$. Therefore, in
the following we may directly assume
$$D_0=X=[0,8] \times [0,4].$$

Let $\Omega=\partial D_0-[u,v]$ be as in Claim 3. Let
$V_{\Omega}=U_{\Omega}-D_0$ be as in (\ref{eq:3-19}). Let $\mathbb{V}$ be as
in (\ref{eq:3-20}). Take an arc $J$ such that $\stackrel{~\circ}J \subset V_\Omega$
and $\partial J=\{u,v\}$. Write $C_1=J \cup \Omega$. Then $C_1$ is a
circle. Let $D_1=D(C_1)$ be the disc in $\mathbb{R}^2$ such that
$\partial D_1=C_1$. Then $D_1 \cap D_0=\overline{\Omega}=\Omega \cup
\{u,v\}$.

\medskip Let $D_2=[u,v] \times [-4, 0]$, and $D_3=D_2 \cup D_0$.
Consider the sphere $\mathbb{R}^{\;2}_{\infty}=\mathbb{R}^2 \cup
\{\infty\}$. Take a homeomorphism $h: \mathbb{R}^{\;2}_{\infty}
\rightarrow \mathbb{R}^{\;2}_{\infty}$ such that
$$h|D_0=\mathrm{id}, \ \ \ \ h(\stackrel{~\circ}D_1)=\mathbb{R}^{\;2}_{\infty}-D_3,$$
and
$$h(\mathbb{R}^{\;2}_{\infty}-D_0-D_1)=\stackrel{~\circ}D_2.$$
Let $\varphi=h \circ f: D_0 \rightarrow \mathbb{R}^{\;2}_{\infty}$.
Then we have
\begin{equation}\label{eq:3-35} \varphi(x)=u,\ \ \varphi(y)=v, \ \ \varphi\big([x,y]\big)=K,
\end{equation}
\begin{equation}\label{eq:3-36} \varphi^{-1}(K_0)=(x,y),\ \ \mathrm{Fix}(\varphi)=\mathrm{Fix}(f),\end{equation}
 and by (\ref{eq:3-22}) we have
\begin{equation}\label{eq:3-37}
\begin{split}\varphi(D_0)=h \circ f(D_0) &\subset
h(\mathbb{R}^2-V_{\Omega}-\mathbb{V}) \subset
h(\mathbb{R}^{\;2}_{\infty}-\stackrel{~\circ}D_1)-\mathbb{V}\\
&= \mathbb{R}^{\;2}_{\infty}-h(\stackrel{~\circ}D_1)-\mathbb{V}=D_2
\cup D_0-\mathbb{V}. \end{split}
\end{equation}
\begin{equation}\label{eq:3-38}
\begin{split}\varphi\big([u,x) \cup (y,v]\big)&= h \circ
f\big([u,x)\cup (y,v]\big) \subset h(\mathbb{R}^2-D_0-V_{\Omega})\\
&\subset h(\mathbb{R}^2-D_0-D_1)=\stackrel{~\circ}D_2.
 \end{split}
\end{equation}
Write $E_1=[u,x] \times [-1,0]$, and $E_2=[y,v] \times [-1,0]$. For
any point $w \in [u,v]$ and $r \in \mathbb{R}$, write $w_r=(w,-r)$.
Then $w_0=w$. Write
$$J_1=[u,x], \ \ J_2=[u,u_1],\ \ J_3=[u_1,x_1], \ \ J_4=[x_1, x],$$
$$J_5=[y,v], \ \ J_6=[y,y_1],\ \ J_7=[y_1,v_1], \ \ J_8=[v_1, v].$$
Then $\partial E_1=J_1 \cup J_2 \cup J_3 \cup J_4$, $\partial
E_2=J_5 \cup J_6 \cup J_7 \cup J_8$.

\medskip For any $c, c' \in \mathbb{R}$ and any $z=(t,r),
z'=(t',r') \in \mathbb{R}^2$, write $cz+c'z'=(ct+c't', cr+c'r')$.
For any $w \in [u,x]$, write $a_{w}=u+x-w$ and $b_w=\min\{w-u,
x-w,1\}$. For any $w \in [y,v]$, write $a_w=y+v-w$ and
$b_w=\min\{w-y, v-w, 1\}$. Let $D_4=D_0 \cup E_1 \cup E_2$. By
(\ref{eq:3-37}) and (\ref{eq:3-38}), we can define a continuous map $\psi: D_4
\rightarrow D_3-\mathbb{V}$ as follows:

\medskip (a)\ Let $\psi|D_0=\varphi$, and let $\psi(w_1)=(a_w,
-1-b_w)$ for any $w \in [u,x] \cup [y,v]$;

\medskip (b)\ For any $w \in [u,x] \cup [y,v]$ and any $r \in
[0,1]$, let $\psi(w_r)=(1-r)\psi(w_0)+r\psi(w_1)$.

\medskip Let $p_2: \mathbb{R}^2 \rightarrow \mathbb{R}$ be the
projection defined by $p_2(z)=r_2$ for any $z=(r_1,r_2) \in
\mathbb{R}^2$. From (\ref{eq:3-34}), (\ref{eq:3-37}) and the definition of $\psi$ we
see that $\psi$ has the following properties:

\medskip (P.1)\ \ $\psi(J_4)=J_2$ with $\psi(x_r)=u_r$ and
$\psi(J_6)=J_8$ with $\psi(y_r)=v_r$ for any $r \in [0,1]$, and
$\psi(u_1)=x_1$ with $\psi(v_1)=y_1$;

\medskip (P.2)\ \ $p_2 \circ \psi(z)<p_2(z)$ for any $z \in E_1
\cup E_2-J_4-J_6-\{u_1,v_1\}$;

\medskip (P.3)\ \ It follows from (P.1) and (P.2) that
$\mathrm{Fix}(\psi) \cap (E_1 \cup E_2)=\emptyset$, and hence
$\mathrm{Fix}(\psi)=\mathrm{Fix}(\varphi)=\mathrm{Fix}(f)$.

\medskip Let $J_9=\psi(J_3)$ and $J_{10}=\psi(J_7)$. Then $J_9
\cap J_3=\{u_1, x_1\}$ and $J_{10} \cap J_7=\{y_1, v_1\}$. Let
$C_3=J_3 \cup J_9$ and $C_4=J_7 \cup J_{10}$. Then $C_3$ and $C_4$
are circles. For $n=3,4$, let $E_n=D(C_n)$ be the disc in
$\mathbb{R}^2$ such that $\partial E_n=C_n$. Write $D_5=D_4 \cup E_3
\cup E_4$. Then $D_4 \subset D_5 \subset D_3$. Let $L=J_9 \cup J_4
\cup [x,y] \cup J_6 \cup J_{10}$. Then $L$ is an arc in $\partial
D_5$. Clearly, there is a retraction $\zeta: D_3 \rightarrow D_5$
such that $\zeta(D_3-D_5) \subset L$. Let $\phi=\zeta \circ \psi:
D_4 \rightarrow D_5$. Since $\psi(L \cap D_4) \subset D_4$, we have
$\mathrm{Fix}(\phi)=\mathrm{Fix}(\psi)$, which with (P.3) implies
$\mathrm{Fix}(\phi)=\mathrm{Fix}(f)$. Let
$$B=B(0,5),\ \ B_1=B \cap \big([-5,-3] \times [-5,5]\big),$$
$$B_2=B \cap \big([-3,3] \times [-5, 5]\big), \ \ B_3=B \cap \big([3,5] \times [-5,5]\big),$$
$$A_1=\{-3\} \times [-4,4], \ \ A_2=\partial B \cap \big([-5,5] \times [-5,-4]\big),$$
$$A_3=\{3\} \times [-4,4], \ \ A_4=\partial B \cap \big([-5,5] \times [4,5]\big).$$
Take a homeomorphism $\varsigma: D_5 \rightarrow B$ such that
$$\varsigma(E_3)=B_1, \ \ \varsigma(J_3)=A_1, \ \ \varsigma(D_4)=B_2,$$
$$\varsigma(L \cap D_4)=A_2, \ \ \varsigma(J_7)=A_3, \mbox{\ \ and\ \ } \varsigma(E_4)=B_3.$$
Let $F=\varsigma \circ \phi \circ \varsigma^{-1}|B_2: B_2
\rightarrow B$. Then
$\mathrm{Fix}(F)=\varsigma\big(\mathrm{Fix}(\phi)\big)$.

\medskip If $\mathrm{Fix}(F)=\emptyset$, then we can define a
map $G: B_2 \rightarrow \partial B_2$ as follows:

\medskip (c)\ \ For any $z \in B_2-A_2-A_4$, there exists a
unique positive number $r_z$ such that $z+r_z\big(z-F(z)\big) \in
\partial B_2$, and then we put $G(z)=z+r_z\big(z-F(z)\big)$;

\medskip (d)\ \ For any $z \in A_2 \cup A_4$, put $G(z)=z$.

\medskip \noindent It is easy to show that $G$ is continuous
at every point $z \in B_2$. On the other hand, by computing we see
that the degree of the map $G|\partial B_2: \partial B_2 \rightarrow
\partial B_2$ is $-1$, which cannot be extended to be a continuous
map from $B_2$ to $\partial B_2$. These lead to a contradiction.
Thus $\mathrm{Fix}(F) \neq \emptyset$, and hence,
$\mathrm{Fix}(f)=\mathrm{Fix}(\phi)=\varsigma^{-1}\big(\mathrm{Fix}(F)\big)
\neq \emptyset$. Similar to the proof of the Claim 5, we have
$\mathrm{Fix}(f) \cap \big(\mathbb{V} \cup K_0 \cup
[u,v]_A\big)=\emptyset$. Therefore, $\mathrm{Fix}(f) \cap
D_0=\mathrm{Fix}(f) \cap W \neq \emptyset$. Claim 9 is proved,
and the proof of Theorem 2.1 is completed. \hfill$\Box$

\section{Outflanking arc theorem}

In this section, we will first introduce a notion of outflanking arc. Then using this notion we establish a fixed point theorem,
called Outflanking Arc Theorem. A part of Theorem \ref{main} is used in the proof of this theorem. At last, we deduce Brouwer's Lemma
from Outflanking Arc Theorem, and construct an example to show the more generality of the theorem than Brouwer's Lemma.

\begin{defin}\label{step arc} Let $Y \subset \mathbb{R}^2$, and $f: Y
\rightarrow \mathbb{R}^2$ be a continuous map. For any $z \in Y$,
define $f^0(z)=z$. For $k=1,2, \ldots$, if $f^{k-1}(z)$ has been
defined and $f^{k-1}(z) \in Y$, then we define
$f^k(z)=f(f^{k-1}(z))$.

\medskip

Let $A$ be a given arc in $Y$ with endpoints $x$ and $u$. For $k
\geqslant 0$, write $u_k=f^k(x)$ if $f^k(x)$ has been defined.

\medskip

(1)\, If $u=u_1=f(x)$ then $A$ is called a {\it \textbf{one step arc
of $f$ from $x$ to $u$}}. If there exists $n \geqslant 2$ such that
$u=u_n=f^n(x)$, $\{u_1, \ldots, u_n\} \subset \mathrm{Int}\ A$, $u_0
\prec u_1 \prec \cdots \prec u_{n-1} \prec u_n$ for some natural
order $\prec$ on $A$, and
\begin{equation} \label{eq:4-1}
f([u_{k-1}, u_k]_A)=[u_k, u_{k+1}]_A \mbox{\ \ for\ \ } k \in
\mathbb{N}_{n-1},
 \end{equation}
then $A$ is called an {\it \textbf{$n$ steps arc of $f$ from $x$ to
$u$}}.

\medskip

(2) For some $n \geqslant 1$, if $A$ is an $n$ steps arc of $f$ from
$x=u_0$ to $u=u_n=f^n(x)$, and there exist $y \in (u_{n-1}, u_n]_A$
and $v \in [x, y)_A$ such that $f|[u_{n-1}, y]_A$ is injective,
$v=f(y)$, and $f|(u_{n-1}, y)_A$ dodges $A$, $f|[y, u]_A$ is moving,
that is,
\begin{equation} \label{eq:4-2}
f((u_{n-1}, y)_A) \cap A=f([y,u]_A) \cap [y,u]_A=\emptyset,
 \end{equation}
then $A$ is called an {\it \textbf{$f$-outflanking $n$ steps arc}},
$y$ is called the {\it \textbf{$f$-outflanking point}} of $A$, and
$v$ is called the {\it \textbf{$n$-outflanked point}} of $A$.

\end{defin}

\begin{rem}\label{throw away} In (2) of Definition \ref{step arc}, if $n \geqslant
2$, and the $f$-outflanked point $v \in [u_k, u_{k+1})_A$ for some
$k \in \mathbb{N}_{n-1}$, then $[u_k, u_n]_A$ is an $f$-outflanking
$n-k$ steps arc from $u_k$ to $u_n$. In this case, we can throw away
$[u_0, u_k)_A$ and replace $A$ by $[u_k, u_n]_A$.
\end{rem}

As an example, we have the following

\begin{lem}  Let $Y \subset \mathbb{R}^2$, $f: Y
\rightarrow \mathbb{R}^2$ be a continuous map, and $A$ be an $n$
steps arc of $f$ from $x=u_0$ to $u=u_n=f^n(x)$, $n \geqslant 2$.
Let $\{u_1, \ldots, u_{n-1}\} \subset \stackrel{\circ}A$ be the same
as in (1) of Definition \ref{step arc}. If $f|A$ is injective, and
$f((u_{n-1}, u_n]_A) \cap A \neq \emptyset$, then $A$ is an
$f$-outflanking $n$-steps arc.
\end{lem}

\begin{proof} Write $A_k=[u_k, u_{k+1}]_A$ for $k \in
\mathbb{N}_{n-1} \cup \{0\}$. Since $f|A$ is injective, for $k \in
\mathbb{N}_{n-1}$, we have
$$f((u_{n-1}, u_n]_A) \cap A_k=f((u_{n-1}, u_n]_A) \cap f(A_{k-1})=f((u_{n-1}, u_n]_A \cap A_{k-1})=\emptyset,$$
which with $f((u_{n-1}, u_n]_A) \cap A \neq \emptyset$ implies
$$f((u_{n-1}, u_n]_A) \cap A_0-\{u_1\} \neq \emptyset.$$
Thus there is $y \in (u_{n-1}, u_n]_A$ such that $f(y) \in [u_0,
u_1)_A \subset [x,y)_A$ and $f((u_{n-1}, y)_A) \cap A =\emptyset$;
the latter equality means that $f|(u_{n-1}, y)_A$ dodges $A$. We
also have
\begin{eqnarray*}
f([y, u_n]_A) \cap [y, u_n]_A &\subset& f((u_{n-1}, u_n]_A) \cap
f((u_{n-2}, u_{n-1}]_A) \\
&=& f((u_{n-1}, u_n]_A \cap (u_{n-2}, u_{n-1}]_A)=\emptyset,
\end{eqnarray*}
which means that $f|[y, u]_A$ is moving. Hence $A$ is an
$f$-outflanking $n$ steps arc. The lemma is proved.
\end{proof}

The following Outflanking Arc Theorem is the main results of this section, which
shows that an $f$-outflanking arc with some technical conditions (mainly a local orientation preserving and injectivity condition, and an
exclusivity condition) implies the existence of a
fixed point.

\medskip

\begin{thm}[Outflanking Arc Theorem] \label{main-2}Let $Y \subset \mathbb{R}^2$, and $f:
Y \rightarrow \mathbb{R}^2$ be a continuous map. Suppose that

\medskip

(1)\,  For some $n \geqslant 1$, there exists an $f$-outflanking
$n$ steps arc $A$ from $x=u_0$ to $u=u_n=f^n(x)$, and there exists a
disc $E \subset Y$ such that $A \cup f(A) \subset
\stackrel{\circ}E$;

\medskip

(2)\,  Let $\{u_1, \ldots, u_{n-1}\} \subset \stackrel{\circ} {A}$  be
the same as in Definition \ref{step arc}, and let $y \in (u_{n-1}, u_n]_A$ be
the $f$-outflanking point of $A$. Suppose that there exists a connected open neighborhood
$U_{n-1}$ of $u_{n-1}$ in $\stackrel{\circ}E$ such that $f|U_{n-1}$
is an orientation preserving injection, and $f|(u_{n-1}, y)_A$ is
exclusive in $E$, that is,
\begin{equation} \label{eq:4-3}
f(E-(u_{n-1}, y)_A) \cap f((u_{n-1}, y)_A)=\emptyset.
 \end{equation}
  Then $f$ has a fixed point in $\stackrel{\circ}E$.
\end{thm}

\begin{proof} Let $v=f(y) \in [x,y)_A$ be the $f$-outflanked
point of $A$. Let
$$K=f([u_{n-1}, y]_A),\hspace{5mm} K_0=f((u_{n-1}, y)_A) \hspace{5mm} \mbox{and}
\hspace{5mm} P=K \cup (v, u)_A.$$ Then $K_0 \subset K \subset E$.
Since $f|[u_{n-1}, y]_A$ is injective, $K$ is an arc, and
$K_0=\stackrel{\circ}K$. By \eqref{eq:4-2}, $K \cap A=\partial
K=\{v, u\}$, $P=K \cup [v,u]_A$ is a circle, and $K \cup A$ is a
$\sigma$-graph (if $v \in (x,y)_A$) or a circle (if $v=x$). Write
$$y_0=u_{n-1} \hspace{5mm} \mbox{and} \hspace{5mm} y_4=y.$$
Take $\{y_2, y_1, y_3\} \subset (y_0, y_4)_A$ such that $v \in
[x,y_2)_A$, $y_1 \in (y_0, y_2)_A$ and $y_3 \in (y_2, y_4)_A$. Note
that we regard $\mathbb{R}$ as a subspace of $\mathbb{R}^2$. By
Remark \ref{top conj}, we may assume that the following properties (P.1)-(P.3)
hold:

\medskip

(P.1) \hspace{30mm} $A=[0,n] \subset \mathbb{R} \subset
\mathbb{R}^2$,
$$u_k=k \mbox{\ \ for\ \ } k \in \mathbb{N}_n \cup \{0\}, \hspace{5mm} y \in (n-1, n],$$
$$f([k-1, k])=[k, k+1] \mbox{\ \ for\ \ } k \in \mathbb{N}_{n-1},$$
and there exists $\varepsilon_0 \in (0, (y-u_{n-1})/4)$ such that
\begin{equation} \label{eq:4-4}
f(r)-r>\varepsilon_0 \mbox{\ \ for any\ \ } r \in [0, n-1];
 \end{equation}

\medskip

(P.2)\, $y_i=u_{n-1} + i \cdot (y-u_{n-1})/4$ for $i \in
\mathbb{N}_4 \cup \{0\}$, and $v \in [0, y_2)$. Write
$$w_0=(n, 0), \ \ w_1=(n+1,0),$$
$$w_2=(n+1, 1), \ \ w_3=(v,1), \ \ w_4=(v,0),$$
$$w_5=(-1,1), \ \ w_6=(-1,-1),\ \ w_7=(v,-1).$$
Then we have

\medskip

(a)\, $f(y_i)=w_i$ for $i \in \mathbb{N}_4 \cup \{0\}$;

\medskip

(b)\, For $i \in \mathbb{N}_3$, $f([y_{i-1}, y_i])=[w_{i-1}, w_i]$,
and $f|[y_{i-1}, y_i]$ is linear;

\medskip

(c)\, If (1) $v=x=0$, or (2) $v \in (0,y)$ and $[0,v) \subset
E-D(P)$, then $f([y_3, y_4])=[w_3, w_4]$;

\medskip

(d)\, If (3) $v \in (0,y)$ and $[0,v) \subset D(P)$, then
$$f([y_3, y_4])=[w_3, w_5] \cup [w_5, w_6] \cup [w_6, w_7] \cup [w_7, w_4].$$

\medskip

(P.3)\, For any real number $a<b$ and $r>0$, write
$$R_r[a,b]=[a,b] \times [-r,r].$$
There exists $$\delta \in \left(0, \frac{\varepsilon_0}{6}\right)
\subset \left(0, \frac{y_1-y_0}{6}\right)$$ such that
$$R_{\delta}[-\delta, n+\delta] \subset \stackrel{\circ}E, \hspace{5mm} B(u_{n-1}, 2\delta) \subset U_{n-1},$$
and
\begin{equation} \label{eq:4-5}
d(f(w), f(z))<\frac{\varepsilon_0}{3} \mbox{\ \ for any\ \ } w, z
\in E \mbox{\ \ with\ \ } d(w,z) \leqslant 2\delta.
 \end{equation}

\medskip

Let $p: \mathbb{R}^2 \rightarrow \mathbb{R}$ be the projection
defined by
$$p(z)=r \mbox{\ \ for any\ \ } z=(r,s) \in \mathbb{R}^2.$$
Then for any $r \in [0, y_1]$ we have $f(r) \in [1, n+1] \subset
\mathbb{R} \subset \mathbb{R}^2$ and hence $p(f(r))=f(r)$. From the
property (P.2) we get
\begin{equation} \label{eq:4-6}
f(r)-r \geqslant 1> \varepsilon_0 \mbox{\ \ for any\ \ } r \in
[y_0,y_1],
 \end{equation}
\begin{equation} \label{eq:4-7}
d(f(r),r) \geqslant 1 \mbox{\ \ for any\ \ } r \in [y_0,y_3],
 \end{equation}
\begin{equation} \label{eq:4-8}
r-p(f(r))>y_3-v>y_3-y_2>\varepsilon_0 \mbox{\ \ for any\ \ } r \in
(y_3,y_4].
 \end{equation}
It follows from \eqref{eq:4-4}, \eqref{eq:4-7} and \eqref{eq:4-8}
that
\begin{equation} \label{eq:4-9}
d(f(r), r)>\varepsilon_0 \mbox{\ \ for any\ \ } r \in [0,y],
 \end{equation}
and from \eqref{eq:4-9} and \eqref{eq:4-5} we obtain
\begin{equation} \label{eq:4-10}
d(f(z),z)>\frac{\varepsilon_0}{3} \mbox{\ \ for any\ \ } z \in
R_{\delta}[-\delta, y+\delta].
 \end{equation}

Let $X=D(P)$. Then $X$ is a disc in $\stackrel{\circ}E$. Let
$W=~\stackrel{\circ}X$. Then $W$ is the unique bounded connected
component of $\mathbb{R}^2-P$. Let $q=u_{n-1}+\delta$. Then $q \in
(u_{n-1}, y_1)$, and
$$f((q-\delta, q + \delta]) \subset f((u_{n-1}, y_1))=(n, n+1).$$
Since $f|[y_0, y_1]$ is linear, we have
$$f(q+\delta)-f(q)=f(q)-f(u_{n-1})=f(q)-n.$$
Take an $\varepsilon \in (0, \delta/2]$ such that
\begin{equation} \label{eq:4-11}
d(f(w), f(z))<f(q)-n
 \end{equation}
for any $w, z \in E$ with $d(w,x) \leqslant \varepsilon$. Let
\begin{equation} \label{eq:4-12}
U_q=B(q, \varepsilon) \cap ([u_{n-1}, y_1] \times [0,
\varepsilon]).
 \end{equation}
Then $U_q$ is a disc, which is a neighborhood of $q$ in $X$, and is
on the left side of the directed arc $[u_{n-1}, y_1]$ from $u_{n-1}$
to $y_1$. Since $U_q \subset U_{n-1}$ and $f|U_{n-1}$ is an
orientation preserving injection, $f(U_q)$ is also a disc on the
left side of the directed arc $[n, n+1]$ from $n$ to $n+1$. By the injectivity of
$f|U_{n-1}$ again, we have
\begin{equation} \label{eq:4-13}
f(U_q-A) \cap f([u_{n-1}, q+\delta])=\emptyset.
 \end{equation}
From \eqref{eq:4-11} and \eqref{eq:4-12} we get
$$f(U_q) \cap (P-f([u_{n-1}, q+\delta]))=\emptyset,$$
which with \eqref{eq:4-13} implies $f(U_q-A)\cap P=\emptyset$, and
hence we obtain
\begin{equation} \label{eq:4-14}
f(U_q-A) \subset W.
 \end{equation}

By Remark \ref{throw away}, we need only discuss the following three cases:

\medskip

{\bf Case 1.}\, $n=1$ and $v=x$. In this case, we have $x=v<y
\leqslant u$. By \eqref{eq:4-2}, we have
\begin{equation} \label{eq:4-15}
K \cap (v, u)_A=f([y, u]_A) \cap [y, u]_A=\emptyset.
 \end{equation}
Thus $f|X$ satisfies the condition (2) or the condition (1) of
Theorem \ref{main}, and hence, $f$ has a fixed point in $W \subset X
\subset \stackrel{\circ}E$.

\medskip

{\bf Case 2.}\, $n=1$ and $v \in (u_0, y_2)$. In this case, we
consider the rectangle $R=R_{\delta}[-\delta, y_3]$, and define a
homeomorphism $h: \mathbb{R}^2 \rightarrow \mathbb{R}^2$ by

\medskip

(h.1)\, $h|(\mathbb{R}^2-\stackrel{\circ}R)=\mathrm{id}$, and
$h(v)=0=u_0$;

\medskip

(h.2)\, For any $z \in \partial R$, $h|[z,v]$ is linear.
\medskip
\\
Let $g=hf: Y \rightarrow \mathbb{R}^2$. Then we have

\medskip

(g.1)\, For any $z \in Y$, if $f(z) \in
\mathbb{R}^2-\stackrel{\circ}R$ then $g(z)=f(z)$;

\medskip

(g.2)\, For any $z \in Y-\stackrel{\circ}R$, if $f(z) \in
\stackrel{\circ}R$ then $g(z) \in \stackrel{\circ}R$, and hence,
both $z \notin \mathrm{Fix}(f)$ and $z \notin \mathrm{Fix}(g)$.

\medskip

It follows from (g.1) and (g.2) that
\begin{equation} \label{eq:4-16}
\mathrm{Fix}(g) \cap (Y-\stackrel{\circ}R)=\mathrm{Fix}(f) \cap
(Y-\stackrel{\circ}R).
 \end{equation}
By the property (P.2), we have
$$d([y_0, y_3], f([y_0, y_3]))=n-y_3 \geqslant y_4-y_3,$$
which with \eqref{eq:4-5} and \eqref{eq:4-10} implies
\begin{equation} \label{eq:4-17}
d(R,f(R))>y_4-y_3-2\delta-\frac{\varepsilon_0}{3}>\frac{y_4-y_3}{2}.
 \end{equation}
Thus $f(R) \subset \mathbb{R}^2-R$, and by (g.1) we get
$$\mathrm{Fix}(g) \cap R=\mathrm{Fix}(f) \cap R=\emptyset,$$
which with \eqref{eq:4-16} implies
$\mathrm{Fix}(g)=\mathrm{Fix}(f)$.

\medskip

Clearly, we have

\medskip

\textbf{Claim 1.}\ {\it $A=[0,1]$ is a $g$-outflanking one step arc
from $x=u_0=0$ to $u=u_1=1$, $y$ is still the $g$-outflanking point
of $A$, $g|(u_0, y)_A$ is also exclusive in $E$, $g|U_{n-1}$ is also
an orientation preserving injection, but the $g$-outflanked point of
$A$ is $u_0$ not $v$}.

\medskip

By Claim 1 and the discussion in Case 1, $g$ has a fixed point in
$\stackrel{\circ}E$, and hence $f$ has a fixed point in
$\stackrel{\circ}E$.

\medskip

{\bf Case 3.}\, $n \geq 2$ and $v \in [u_0, u_1)$. In this case, we
consider the rectangle $R_1=R_{\delta}[-\delta, y_1]$, and define a
homeomorphism $h_1: \mathbb{R}^2 \rightarrow \mathbb{R}^2$ by

\medskip

($\mathrm{h}_1$.1)\,
$h|(\mathbb{R}^2-\stackrel{\circ}R_1)=\mathrm{id}$, and
$h_1(v)=0=u_{n-1}$;

\medskip

($\mathrm{h}_1$.2)\, For any $z \in \partial R_1$, $h_1|[z,v]$ is
linear. \medskip
\\
Let $g_1=h_1f: Y \rightarrow \mathbb{R}^2$. Similar to
\eqref{eq:4-16}, we have
\begin{equation} \label{eq:4-18}
\mathrm{Fix}(g_1) \cap (Y-\stackrel{\circ}R_1)=\mathrm{Fix}(f) \cap
(Y-\stackrel{\circ}R_1).
 \end{equation}
For any $z \in R_1$, from \eqref{eq:4-4}, \eqref{eq:4-5},
\eqref{eq:4-6} and \eqref{eq:4-10} we obtain
\begin{equation} \label{eq:4-19}
p(f(z))-p(z)>y_4-y_3-2\delta-2\delta-\frac{\varepsilon_0}{3}>\frac{y_4-y_3}{2}.
 \end{equation}
Note that $p(g_1(z))=p(f(z))$ if $f(z) \in
\mathbb{R}^2-\stackrel{\circ}R_1$, and $p(g_1(z))>p(f(z))$ if $f(z)
\in \stackrel{\circ}R_1$. From \eqref{eq:4-19}, we get
$$p(g_1(z))-p(z)>\frac{y_4-y_3}{2} \mbox{\ \ for any\ \ } z \in R_1,$$
which with \eqref{eq:4-19} and \eqref{eq:4-18} implies
$$\mathrm{Fix}(g_1)=\mathrm{Fix}(g_1) \cap (Y-\stackrel{\circ}R_1)=\mathrm{Fix}(f) \cap (Y-\stackrel{\circ}R_1)=\mathrm{Fix}(f).$$

Similar to Claim 1, we have

\medskip

\textbf{Claim 2.}\ {\it $[u_{n-1},u_n]$ is a $g_1$-outflanking one
step arc from $u_{n-1}=n-1$ to $u=u_n=n$, $y$ is still the
$g_1$-outflanking point of $A$, $g_1|(u_{n-1}, y)_A$ is also
exclusive in $E$, $g_1|U_{n-1}$ is also an orientation preserving
injection, but the $g_1$-outflanked point of $A$ is $u_{n-1}$ not
$v$}.

\medskip

By Claim 2 and the discussion in Case 1, $g_1$ has a fixed point in
$\stackrel{\circ}E$, and hence $f$ has a fixed point in
$\stackrel{\circ}E$. Theorem \ref{main-2} is proved.
\end{proof}

\begin{thm} \label{main3} Let $Y \subset \mathbb{R}^2$, and $f:
Y \rightarrow \mathbb{R}^2$ be a continuous map. Suppose that $f$
has an $m$-periodic point $x$ for some $m \in \mathbb{N}-\{1\}$. For
any $k \in \mathbb{N} \cup \{0\}$, write $u_k=f^k(x)$. Suppose that
there is a connected open set $V_0$ in $\mathbb{R}^2$ such that
$\{u_0, u_1\} \subset V_0 \subset Y-\mathrm{Fix}(f)$ and $f^k(V_0)
\subset Y$ for $k \in \mathbb{N}_{m-1}$. For $k \in \mathbb{N}_m$,
write $V_k=f^k(V_0)$. Let $V=\bigcup \{V_k: k \in \mathbb{N}_{m-1}
\cup \{0\}\}$.

\medskip

(1)\,  If $f|V$ is injective, then there exist $n \in
\mathbb{N}_{m-1}$ and an $f$-outflanking $n$ steps arc $A$ from
$u_0$ to $u_n$ such that $A \subset V$.

\medskip

(2)\,  Further, let $y \in (u_{n-1}, u_n]_A$ be the
$f$-outflanking point of $A$. If $f|V$ is an orientation preserving
injection, and there exists a disc $E$ in $Y$ such that $A \cup f(A)
\subset \stackrel{\circ}E$ and $f|(u_{n-1}, y)_A$ is exclusive in
$E$, then $f$ has a fixed point in $\stackrel{\circ}E$.

\end{thm}

\begin{proof} (1) Take a disc $D \subset V_0$ such that $\{u_0,
u_1\} \subset \stackrel{\circ}D$. By Remark \ref{top conj}, we may assume that
$$D=[-1,2] \times [-1,1], \mbox{\ \ and\ \ } x=u_0=0, u_1=1.$$
For any $t \in (0,1]$, write $Q_t=[-t,t]^2$. Then there exists $b
\in (0,1)$ such that
$$Q_t \cap f(Q_t)=\emptyset \mbox{\ \ for any\ \ } t \in (0,b),$$
and
$$f(Q_b) \cap Q_b=f(\partial Q_b) \cap \partial Q_b \neq \emptyset.$$
Take a point $w \in \partial Q_b \cap f(\partial Q_b)$. Then there
is $w' \in \partial Q_b$ such that $f(w')=w$. Write $L=f([x, w'])$.
Then $L \subset f(Q_b) \subset f(D)$. Since $f|D$ is injective,
$f(Q_b)$ is a disc, $L$ is an arc, $\partial L=\{w, u_1\}$, and $(w,
u_1]_L \subset f(\stackrel{\circ}Q_b)$. Since $D \cap
\mathrm{Fix}(f)=\emptyset$, we have $w' \neq w$ and $f(w) \in
f(\partial Q_b)-\{w\}$. Let $J=[x,w] \cup L$. Then $J$ is an arc,
$\partial J=\{x, u_1\}$, $L=[w, u_1]_J$, and $J \subset D \cup
f(D)$. From
$$f((x, w]) \cap L=f((x,w]) \cap f([x,w'])=\emptyset$$
and
$$f((x,w]) \cap [x,w] \subset (f(\stackrel{\circ}Q_b) \cup \{f(w)\})\cap [x,w]=\emptyset,$$
we get $f((x,w]) \cap J=\emptyset$, which with
$$f(L) \cap L=f([w,u_1]_J) \cap f([x,w'])=f([w,u_1]_J \cap [x,w'])=\emptyset$$
implies
\begin{equation} \label{eq:4-20}
f((x,u_1]_J) \cap J=f((w, u_1]_L) \cap J=f((w,u_1]_L) \cap [x,w).
 \end{equation}
By \eqref{eq:4-20}, we have

\medskip

\textbf{Claim 1.}\, If $f((x, u_1]_J) \cap J \neq \emptyset$,
then there exist $y \in (w, u_1]_J$ and $v \in [x,w)_J$ such that
$f(y)=v$ and $f((x,y)_J) \cap J=\emptyset$.

\medskip

\textbf{Case 1.}\, We first consider the case $f((x,u_1]_J) \cap J
\neq \emptyset$. In this case, let $y$ and $v$ be the same as in
Claim 1. By Claim 1, $f|(x,y)_J$ dodges $[v,u_1]_J$ since
$$f((x,y)_J) \cap [v,u_1]_J \subset f((x, u_1]_J) \cap J =\emptyset,$$
and $f|[y, u_1]_J$ is moving since
$$f([y,u_1]_J) \cap [y, u_1]_J \subset f(L) \cap L =\emptyset.$$
Therefore, $J$ is an $f$-outflanking one step arc from $x=u_0$ to
$u_1=f(x)$ and in this case we can take $A=J$.

\medskip

\textbf{Case 2.}\, In the following we consider the case
$f((x,u_1]_J) \cap J=\emptyset$. In this case, we have
$$f(J) \cap J=f(\{x\} \cup (x,u_1]_J) \cap J=\{f(x)\}=\{u_1\}.$$
We also have $m \geqslant 3$, since if $m=2$ then $f(u_1)=x$, which
will lead to $f((x,u_1]_J) \cap J \supset \{x\} \neq \emptyset$.

\medskip

For each $i \in \mathbb{N}_m \cup \{0\}$, write $D_i=f^i(D)$. For
each $i \in \mathbb{N}_{m-1} \cup \{0\}$, write $J(i)=J_i=f^i(J)$.
Since $f|V$ is injective, the following properties hold:

\medskip

(P.1)\, For each $i \in \mathbb{N}_m \cup \{0\}$, $D_i$ is a disc.

\medskip

(P.2)\, For each $i \in \mathbb{N}_{m-1} \cup \{0\}$, $J_i$ is an
arc, $\partial J_n=\{u_i, u_{i+1}\}$, $J_i \subset D_i \cup D_{i+1}$,
and $D_i \subset V \subset Y$.

\medskip

(P.3)\, For each $i \in \mathbb{N}_{m-1}$, $J_i \cap
J_{i-1}=\{u_i\}$.

\medskip

(P.4)\, For $i, j \in \mathbb{N}_{m-2}$ with $i+j<m$,
$$(u_i, u_{i+1}]_{J(i)} \cap J=f^i((x,u_1]_J) \cap J=\emptyset$$
if and only if $f^{i+j}((x,u_1]_J) \cap f^j(J)=\emptyset$.

\medskip

Noting that $u_m=u_0$, we have
\begin{equation} \label{eq:4-21}
f^{m-1}((x,u_1]_J) \cap J = (u_{m-1}, u_m]_{J(m-1)} \cap J \supset
\{u_0\} \neq \emptyset.
 \end{equation}
By (P.4), (P.3) and \eqref{eq:4-21} we get

\medskip

(P.5)\, There exists $n \in \mathbb{N}_{m-1}-\{1\}$ such that both
$\{(u_i, u_{i+1}]_{J(i)}: i \in \mathbb{N}_n\}$ and $\{(u_{i-1},
u_i]_{J(i-1)}: i \in \mathbb{N}_n\}$ are families of pairwise
disjoint semi-open arcs, but
$$(u_n, u_{n+1}]_{J(n)} \cap [x, u_1)_J \neq \emptyset,$$
and there exists $y \in (u_{n-1}, u_n]_{J(n-1)}$ such that $f(y) \in
[x, u_1)_J$ and
\begin{equation} \label{eq:4-22}
f([u_{n-1}, y)_{J(n-1)}) \cap J_i-\{u_{i+1}\}=\emptyset \mbox{\ \
for\ \ } 0 \leqslant i \leqslant n-1.
 \end{equation}

\medskip

Let $A=\bigcup\{(u_{i-1}, u_i]_{J(i-1)}: i \in \mathbb{N}_n\}$. It
follows from (P.5) and (P.3) that $A$ is an arc, $\partial A=\{x,
u_n\}$, and $\{u_1, u_2, \ldots, u_{n-1}\} \subset
\stackrel{\circ}A$. From \eqref{eq:4-22}, we get
$$f((u_{n-1}, y)_{J(n-1)}) \cap A = \emptyset.$$
Thus $f|(u_{n-1}, y)_A$ dodges $A$. From
$$f([y, u_n]_A) \cap [y, u_n]_A \subset f([y, u_n]_{J(n-1)}) \cap f(J_{n-2})=\emptyset,$$
we see that $f|[y, u_n]_A$ is moving. Therefore, $A$ is an
$f$-outflanking $n$ steps arc, and $y$ is the $f$-outflanking point
of $A$.

\medskip

(2)\, By the conclusion of (1), if $f$ satisfies the conditions in
(1) and (2) of Theorem \ref{main3}, then $f$ satisfies the all conditions
in Theorem \ref{main-2}. Hence, by Theorem \ref{main-2}, $f$ has a fixed point in
$\stackrel{\circ}E$. Theorem \ref{main3} is proved.
\end{proof}
\medskip

Clearly, if $Y=\mathbb{R}^2$ and $f: Y \rightarrow \mathbb{R}^2$ is
an orientation preserving homeomorphism having an $m$-periodic point
$x$ for some $m \in \mathbb{N}-\{1\}$, then $f$ satisfies the all
conditions in Theorem \ref{main3}. Thus we derive directly

\begin{thm}[Brouwer's Lemma]  Let $f: \mathbb{R}^2
\rightarrow \mathbb{R}^2$ be an orientation preserving
homeomorphism. If $f$ has a periodic point, then $f$ has a fixed
point\;\!.
\end{thm}

Let
\begin{eqnarray*}
\mathbf{F} &=& \{f: f \mbox{\ is an orientation preserving
homeomorphism from\ } \mathbb{R}^2 \mbox{\ to\ } \mathbb{R}^2\},\\
\mathbf{F}_P &=& \{f \in \mathbf{F}: f \mbox{\ has a periodic point
with period\ } \geqslant 2\},\\
\mathbf{F}_O &=& \{f \in \mathbf{F}: f \mbox{\ has an
$f$-outflanking $n$ steps arc for some $n \geqslant 1$}\}.
\end{eqnarray*}
From Theorem \ref{main3}, we see that $\mathbf{F}_P \subset \mathbf{F}_O$.
The following example indicates that $(\mathbf{F}_O-\mathbf{F}_P)\cap\mathbf{F}
\neq \emptyset$.

\begin{exa} Let $n \in \mathbb{N}$, and let $\beta \in
\mathbb{R}$ such that $\beta/\pi$ is an irrational number and $n
\beta<2\pi<(n+1)\beta$. Write $\lambda=2^{1/(n+1)}$. Define $f:
\mathbb{C} \rightarrow \mathbb{C}$ by
$$f(z)=\lambda \mathrm{e}^{\beta \mathrm{i}} \cdot z \mbox{\ \ for any\ \ } z \in \mathbb{C}.$$
Then $f \in \mathbf{F}-\mathbf{F}_P$.
\end{exa}

Write $b_0=0$, $b_1=(2\pi-n\beta)/2$, $b_2=2b_1$, $b_3=\beta$. Then
$b_0<b_1<b_2<b_3$. Define a function $\varphi: [0, \beta]
\rightarrow \mathbb{R}$ and an imbedding $\eta: [0, \beta]
\rightarrow \mathbb{C}$ by

\medskip

(i)\, $\varphi(b_0)=\varphi(b_1)=1$, $\varphi(b_2)=1/2$, and
$\varphi(b_3)=\lambda$;

\medskip

(ii)\, $\varphi|[b_{i-1}, b_i]$ is linear, for $i \in \mathbb{N}_3$;

\medskip

(iii)\, $\eta(\theta)=\varphi(\theta) \cdot \mathrm{e}^{\theta
\mathrm{i}}$, for any $\theta \in [0, \beta]$.

\medskip

Write $x=u_0=\eta(0)$, $u_1=\eta(\beta)$, and $A_0=\eta([0,
\beta])$. Then $A_0$ is an arc, $\partial A_0=\{x, u_1\}$, and
$u_1=f(x)$. For $i \in \mathbb{N}$, write $u_i=f^i(x)$, and
$A(i)=A_i=f^i(A_0)$. Then $A_i$ is an arc and $\partial A_i=\{u_i,
u_{i+1}\}$. Let $A=\bigcup\{A_i: i \in \mathbb{N}_{n-1} \cup
\{0\}\}$. Then $A$ is an $n$ steps arc. Let $y=f^n(\eta(b_2))$. Then
$y \in (u_{n-1}, u_n)_{A(n-1)}$. Clearly, $f(y)=\lambda^{n+1} \cdot
1/2=1=u_0$, $f|[u_{n-1},y]_{A(n-1)}$ is injective, $f|(u_{n-1},
y)_A$ dodges $A$, and $f|[y, u_n]_A$ is moving. Thus $A$ is an
$f$-outflanking $n$ steps arc, and hence we have
$\mathbf{F}_O-\mathbf{F}_P \supset \{f\} \neq \emptyset$.

\subsection*{Acknowledgements}
Jiehua Mai, Kesong Yan, and Fanping Zeng  are supported by NNSF of China (Grant No.
12261006) and  NSF of Guangxi Province (Grant No.
2018GXNSFFA281008);  Kesong Yan is also supported by NNSF of China
(Grant No. 12171175) and Project of Guangxi First Class Disciplines
of Statistics and Guangxi Key Laboratory of Quantity Economics; Enhui Shi is supported by NNSF of China (Grant No. 12271388).


\end{document}